\documentclass[7 pt]{article}   
\usepackage{t1enc}
\usepackage{lmodern}
\usepackage[T1]{fontenc}
\usepackage[english]{babel}
\usepackage{graphicx}
\usepackage{tocloft}
\usepackage{yhmath}
\usepackage{amsmath}
\usepackage{textcomp}
\usepackage{ytableau}
\usepackage{hyperref}
\usepackage{mathtools}
\usepackage{tensor}
\usepackage{relsize}
\usepackage[new]{old-arrows}
\usepackage{tikz}
\usepackage{tikz-cd}
\usepackage{faktor}
\usepackage{enumitem}
\usetikzlibrary{matrix}

\usepackage{amsmath,amsthm, amsfonts, amssymb, amscd, mathrsfs, graphicx, graphpap, curves, color}
\usepackage[all]{xy}
\usepackage{amsmath}

\theoremstyle{plain}
\newtheorem{theo}{Theorem}[section]
\newtheorem{lemma}[theo]{Lemma}
\newtheorem{prop}[theo]{Proposition}

\newtheorem{cor}[theo]{Corollary}

\theoremstyle{definition}

\newtheorem{remark}[theo]{Remark}
\newtheorem{question}{Question}
\newtheorem{definition}[theo]{Definition}

\newtheorem{example}[theo]{Example}

\newtheorem{algorithm}[theo]{Algorithm}

\newcommand{\Sym}{\operatorname{Sym}}
\newtheorem*{lemma*}{Lemma}
\newcommand{\superwedge}{\mathlarger{\mathlarger{\wedge}}}

\title{Schur apolarity}
\author{Reynaldo Staffolani}
\date{}

\begin{document}
\maketitle

\begin{abstract}
Inspired by the classic apolarity theory of symmetric tensors, the aim of this paper is to introduce the Schur apolarity theory, i.e. an apolarity for any irreducible representation of the special linear group $SL(V)$. This allows to describe decompositions of structured tensors whose elementary elements are tensors that represent flags of the vector space $V$. The main result is the Schur apolarity lemma which is the analogous of the apolarity lemma of symmetric apolarity theory. Eventually we study the rank tensors of low border rank related to specific varieties giving rise also to simple algorithms.
\end{abstract}
\section*{Introduction} \setcounter{equation}{0}

Tensors are multilinear objects ubiquitous in the applications in different flavours, see e.g. \cite{allman2008phylogenetic, bernardi2012algebraic, bernardi2021high, comon2014tensors}. They can be seen as elements of $V_1 \otimes \dots \otimes V_d$, where any $V_i$ is a vector space of finite dimension over some field $\mathbb{K}$. Being provided with an additive structure, one particular interest is {\it tensor decomposition}, i.e. additive decompositions of tensors into sums of elementary objects, often referred to as tensors of {\it rank 1}. Throughout all this paper we assume that $\mathbb{K}$ is algebraically closed and of characteristic $0$.

\noindent {\it Structured} tensors are tensors with prescribed symmetries between the factors of the tensor product.  Even in this case we can talk about structured tensors of {\it structured rank} $1$ which are in some sense the most elementary tensors with that structure. 

\noindent Some examples are {\it symmetric} tensors $\Sym^d V$, i.e. tensors invariant under permutations of the factors, and {\it skew-symmetric} tensors $\superwedge^d V$, i.e. tensors invariant under permutations of the factors up to the sign of the permutation. In these instances the tensors of {\it structured rank} $1$ are determined by the underlying geometry. Indeed, since both $\Sym^d V$ and $\superwedge^d V$ are irreducible representations of the group $SL(V)$, one can consider the highest weight vector of both of them and its orbit in the projectivization under the group action. These orbits turn out to be projective varieties which are Veronese and Grassmann varieties respectively. The (skew-)symmetric tensors of (skew-)symmetric rank $1$ are the points of the Veronese (Grassmann) variety. In the symmetric case, due to the canonical identification of $\Sym^d V$ with the vector space of homogeneous polynomials of degree $d$ in $\dim V$ variables, the symmetric rank $1$ tensors are $d$ powers of linear forms $l^d$, with $l \in V$. In the skew-symmetric case, the tensors of skew-symmetric rank $1$ look like $v_1 \wedge \dots \wedge v_d$ for some $v_1,\dots,v_d \in V$. \smallskip

\noindent The irreducible representations of $SL(V)$ are known and usually described as {\it Schur modules}, as defined in \cite{fulton2013representation}. The respective minimal orbit inside their projectivization is in general a {\it Flag variety}, and tensors of structured rank $1$ in the general context will represent flags, possibly partial, of the vector space $V$. \smallskip

Coming back to (skew-)symmetric tensors, clearly one would like to compute the (skew-)symmetric rank of any (skew-)symmetric tensor. In both cases but chronologically first in the symmetric, see \cite{iarrobino1999power}, and after for the other one, see \cite{arrondo2021skew}, {\it apolarity theories} have been developed to compute the ranks of such tensors. Even though they may seem different, these two theories share a greatest common idea which is the {\it evaluation}. Starting from the usual contraction $V^* \times V \longrightarrow \mathbb{K}$, one gets respectively the maps

$$ \Sym^d V \otimes \Sym^e V^* \longrightarrow \Sym^{d-e} V\quad \text{and} \quad \superwedge^d V \otimes \superwedge^e V^* \longrightarrow \superwedge^{d-e} V$$
called {\it apolarity actions}. Consequently, given any (skew-)symmetric tensor, one can compute its annihilator, i.e. any element in the dual world, symmetric or skew-symmetric respectively, that kills the given tensor in any degree via the suitable apolarity action. Such annihilator turns out to be an ideal in the symmetric or exterior algebra of $V^*$ respectively. In both cases, given a tensor of (skew-)symmetric rank $1$, there is attached to it an {\it ideal} generated in the respective symmetric or exterior algebra. The ideal of multiple tensors of rank $1$ is just the intersection of the respective ideals in both cases. The key result in both theories is the {\it apolarity lemma} which has the following simple statement. Finding a decomposition of a (skew-)symmetric tensor into a sum of rank $1$ (skew-)symmetric tensors is equivalent to find the inclusion of the ideal of the rank $1$ (skew-)symmetric tensors involved in the decomposition inside the annihilator of the (skew-)symmetric tensor we would like to decompose. It follows this second question. Remark that if a (skew-)symmetric tensor admits more than one decomposition, the apolarity lemma implies that the ideals associated to all such decompositions are contained in the annihilator.

\begin{question}
Is it possible to define an apolarity theory for any other irreducible representation of $SL(V)$?
\end{question}

The main motivation of this document is to answer this question. We will present a suitable apolarity action which will be called {\it Schur apolarity action}. If we denote with $\mathbb{S}_{\lambda} V$ the Schur module determined by the partition $\lambda$, it is the map

$$ \mathbb{S}_{\lambda} V \otimes \mathbb{S}_{\mu} V^* \longrightarrow \mathbb{S}_{\lambda / \mu} V $$
where $\mathbb{S}_{\lambda / \mu} V$ is called {\it skew Schur module}, cf. Definition \ref{skewSchurModule} of this article. Several examples as well as an analogous {\it Schur apolarity lemma} are provided. \smallskip

The content of this article is original but it is worth noting the strong connection of the Schur apolarity with the {\it non-abelian apolarity}. Introduced for the first time in \cite{landsberg2013equations} to seek new sets of equations of secant varieties of Veronese varieties, it is an apolarity which implements vector bundles techniques for any projective variety $X$. More specifically, the non-abelian apolarity action has always in its codomain an irreducible representation. On the contrary in the Schur apolarity case, the skew Schur module is in general reducible. Hence we may think of the Schur apolarity as a slight generalization of the non-abelian one. Another apolarity theory which is worth noting is described in \cite{galkazka2016multigraded} for toric varieties, that features also the use of Cox rings. Formerly the special case in which the toric variety is a Segre-Veronese variety has been introduced in \cite{catalisano2002ranks} as {\it multigraded apolarity}.

\smallskip

The article is organized as follows. In Section \ref{primasez} we recall basic definitions needed to develop the theory. In Section \ref{secondasez} we give a description of the maps $\mathbb{S}_{\lambda} V \otimes \mathbb{S}_{\mu} V \longrightarrow \mathbb{S}_{\nu}V$ which turn out to be useful in the Schur apolarity setting. In Section \ref{terzasez} we introduce the Schur apolarity action. In Section \ref{quartasez} we state and prove the Schur apolarity lemma. In Section \ref{quintasez} we give a description of the ranks of certain linear maps induced by the apolarity action when a structured tensor has structured rank $1$. In Section \ref{sestasez} we investigate the secant variety of the Flag variety of $\mathbb{C}^1 \subset \mathbb{C}^k \subset \mathbb{C}^n$ together with an algorithm which distinguishes tensors of the secant variety with different ranks.
\bigskip

\section{Notation and basic definitions} \label{primasez} \setcounter{equation}{0} \medskip

Let $V$ be a vector space over a field $\mathbb{K}$ of dimension $\operatorname{dim}(V) = n < \infty$. In the following $\mathbb{K}$ is always intended algebraically closed and of characteristic $0$. We adopt the following notation for the algebras:
\begin{enumerate}
\item[$\bullet$] $\Sym^{\bullet}{V} = \bigoplus_{k\geq 0} \Sym^k V$ is the {\it symmetric} algebra, i.e. the algebra of polynomials in $n$ variables;
\item[$\bullet$] $\superwedge^{\bullet} V = \bigoplus_{k \geq 0} \superwedge^k V$ is the {\it exterior} algebra.
\end{enumerate}

\noindent We will indicate with $V^*$ the dual vector space of $V$, i.e. the vector space of linear functional defined on $V$. Remark that this space $V^*$ defines an action $\operatorname{tr} \colon V \times V^* \to \mathbb{K}$  on $V$ called {\it trace}. 

\noindent The group $SL(V)$ is the group of automorphisms of $V$ inducing the identity on the space $\superwedge^n V$. Such a group defines a transitive action on $V$
\begin{align*} SL(V) &\times V\ \to\ V \\ (&g , v) \longmapsto g(v).  \end{align*}
This action can be extended to $V^{\otimes d}$ just setting 

$$g \cdot (v_1 \otimes \dots \otimes v_d) = g(v_1) \otimes \dots \otimes g(v_d).$$

\noindent The symmetric group $\mathfrak{S}_d$ acts on $V^{\otimes d}$ by permuting the factors of the tensor product. It can be easily seen that the actions of $SL(V)$ and $\mathfrak{S}_d$ commutes. These standard facts and their consequences in terms of representations are classical, (cf. \cite{fulton2013representation}).
\medskip

In order to develop an apolarity theory for any irreducible representation of the special linear group we need to set some notation and basic definitions. The facts we are going to recall are borrowed from \cite{fulton2013representation}, \cite{fulton1997young}, \cite{sturmfels2008algorithms} and \cite{procesi2007lie}. \smallskip

In the following we regard the irreducible representations of $SL(n)$ as Schur modules. They are denoted with $\mathbb{S}_{\lambda}V$ where $\lambda = (\lambda_1,\dots,\lambda_k)$, with $k < n$, is a partition. The number $k$ is the {\it length} of the partition. We denote with $|\lambda| := \lambda_1 + \dots + \lambda_k$ the number partitioned by $\lambda$. Sometimes we may also write $\lambda \vdash |\lambda|$ to underline that $\lambda$ is a partition of $|\lambda|$. We follow \cite{fulton2013representation} for a construction of this representations. 

\noindent Fix a partition $\lambda$ of the integer $d$ whose length is less than $\dim(V)$. We can draw its {\it Young diagram} placing $\lambda_1$ boxes in a row, $\lambda_2$ boxes below it and so on, where all the rows are left justified. A filling of positive integers will turn the diagram into a {\it tableau of shape} $\lambda$. A tableau of shape $\lambda$ is said {\it semistandard} if reading from left to right the sequences of integers are weakly increasing, while from top to bottom the sequences are strictly increasing. We will abbreviate it with just {\it sstd} tableau. A sstd tableau of shape $\lambda$ is called {\it standard} if also the row sequences are strictly increasing and there are no repetitions. In this case we will simply write {\it std} tableau. Let $T$ be a std tableau of shape $\lambda$ with entries in $\{1,\dots,d\}$. The {\it Young symmetrizer associated to $\lambda$ and $T$} is the endomorphism $c_{\lambda}^T$ of $V^{\otimes d}$ which sends the tensor $v_1 \otimes \dots \otimes v_d$ to

\begin{equation} \label{YoungSymm} \sum_{\tau \in C_{\lambda}} \sum_{\sigma \in R_{\lambda}} \operatorname{sign}(\tau) v_{\tau(\sigma(1))} \otimes \dots \otimes v_{\tau(\sigma(d))} \end{equation}

\noindent where $C_{\lambda}$, $R_{\lambda}$ respectively, is the subgroup of $\mathfrak{S}_d$ of permutations which fix the content of every column, of every row respectively, of $T$. The image of the Young symmetrizer 

$$\mathbb{S}_{\lambda}^T V := c_{\lambda}^T (V^{\otimes d}) $$

\noindent is called {\it Schur module} associated to $\lambda$. It is easy to see that any two Schur modules that are images of $c_{\lambda}^T$ and $c_{\lambda}^{T'}$, with $T$ and $T'$ two different std tableaux of shape $\lambda$, are isomorphic. Hence we drop the $T$ in the notation and we write only $\mathbb{S}_{\lambda} V$. It can be proved that Schur modules are irreducible representations of $SL(n)$ via the induced action of the group, and they are completely determined by the partition $\lambda$, cf. \cite[p. 77]{fulton2013representation}. From the construction of such modules we have the inclusion

$$ \mathbb{S}_{\lambda} V \subset \superwedge^{\lambda_1'} V \otimes \dots \otimes \superwedge^{\lambda_h'} V =: \superwedge_{\lambda'} V. $$

\noindent where $\lambda'$ is the {\it conjugate} partition to $\lambda$, i.e. obtained transposing the diagram of $\lambda$ as if it were a matrix. \smallskip

We would like to give a little more explicit insight of the elements of such modules since we will need it in the following. We recall briefly the construction given in \cite{sturmfels2008algorithms} of a basis of these spaces in terms of sstd tableaux. The Schur-Weyl duality, see \cite[Theorem 6.4.5.2, p. 148]{Landsberg2011TensorsGA}, gives the isomorphism

\begin{equation} \label{SWduality} V^{\otimes d} \simeq \bigoplus_{\lambda\ \vdash\ d} \mathbb{S}_{\lambda} V^{\oplus m_{\lambda}} \end{equation}

\noindent where $m_{\lambda}$ is the number of std tableaux of shape $\lambda$ with entries in $\{1,\dots,d\}$. For example $m_{(2,1)} = 2$ since we have the two std tableaux

\begin{equation} \label{EsT1} T_1 = \begin{ytableau} 1 & 2 \\ 3 \end{ytableau}\ , \quad T_2 =  \begin{ytableau} 1 & 3 \\ 2 \end{ytableau}\ . \end{equation}

\noindent Now fix a module $\mathbb{S}_{\lambda} V$ together with a std tableaux $T$ of shape $\lambda$, where $\lambda \vdash d$. Fix also a basis $v_1,\dots,v_n$ of $V$ and consider a sstd tableaux $S$ of shape $\lambda$. The pair $(T,S)$, regarded in \cite{sturmfels2008algorithms} as {\it bitableau}, describes an element of $\mathbb{S}_{\lambda}V$ in the following way. At first build the element 
$$v_{(T,S)} := v_{i_1} \otimes \dots \otimes v_{i_d},$$
where $v_{i_j} = v_k$ if there is a $k$ in the box of $S$ corresponding to the box of $T$ in which a $j$ appears. We drop the $T$ in $v_{(T,S)}$ if the choice of $T$ is clear. After that, one applies the Young symmetrizer $c_{\lambda}^T$ to the element $v_{(T,S)}$. For example, if $\lambda =(2,1)$, consider the std tableau $T_1$ in \eqref{EsT1} and the sstd tableau

\begin{center} S = \begin{ytableau} 1 & 1 \\ 2 \end{ytableau} . \end{center}

\noindent Then $c_{\lambda}(v_S) = c_{\lambda}(v_1 \otimes v_1 \otimes v_2) = 2 \cdot v_1 \wedge v_2 \otimes v_1$, which we may represent pictorially as

$$ c_{\lambda}(v_S) = 2 \cdot \left ( \begin{ytableau} 1 & 1 \\ 2 \end{ytableau} -  \begin{ytableau} 2 & 1 \\ 1 \end{ytableau} \right )$$

\noindent where in this particular instance the tableaux represent tensor products of vectors with the order prescribed by $T_1$. We may use sometimes this notation.

\noindent  As a consequence of \cite[Theorem 4.1.12, p. 142]{sturmfels2008algorithms}, one has the following result.

\begin{prop} \label{BaseSchur}
The set 
$$\{ c_{\lambda}^T (v_{(T,S)})\ :\ S\ \text{is a sstd tableau of shape}\ \lambda \}$$
 is a basis of the module $\mathbb{S}_{\lambda} V$. 
\end{prop} 

\noindent As already told, choosing different std tableaux of the same shape give rise to isomorphic Schur modules. The only difference between them is the way they embed in $V^{\otimes |\lambda|}$. Let us see an example.

\begin{example}
Let $\lambda = (2,1)$ and consider the std tableau

$$T_1 = \begin{ytableau} 1 & 2 \\ 3 \end{ytableau}\ ,\ T_2 = \begin{ytableau} 1 & 3 \\ 2 \end{ytableau}$$

\noindent and the sstd tableau

$$S = \begin{ytableau} 1 & 1 \\ 2 \end{ytableau}\ .$$

\noindent We get that 

$$ v_{(T_1,S)} = v_1 \otimes v_1 \otimes v_2,\ v_{(T_2,S)} = v_1 \otimes v_2 \otimes v_1. $$

\noindent Applying the respective Young symmetrizers we get that 

$$ c_{\lambda}^{T_1} (v_{(T_1,S)}) = 2 \cdot (v_1 \otimes v_1 \otimes v_2 - v_2 \otimes v_1 \otimes v_1) = 2 \cdot (v_1 \wedge v_2) \otimes v_1, $$
$$ c_{\lambda}^{T_2} (v_{(T_2,S)}) = 2 \cdot (v_1 \otimes v_2 \otimes v_1 - v_2 \otimes v_1 \otimes v_1) = 2 \cdot (v_1 \wedge v_2) \otimes v_1. $$
\end{example}

\noindent It is clear that we get the same element of $\superwedge^2 V \otimes V$. However it is clear that it embeds in a different way in $V^{\otimes 3}$ under the specific choice of the std tableau of shape $\lambda$. Since we are interested only in the elements of $\superwedge^2 V \otimes V$, we will ignore the way it embeds in $V^{\otimes d}$. For this reason we reduce to work in the vector space

$$ \mathbb{S}^{\bullet} V := \bigoplus_{\lambda} \mathbb{S}_{\lambda} V$$
where roughly every Schur module appears exactly once for each partition $\lambda$. One can see that such a space can also be obtained as the quotient

$$ \mathbb{S}^{\bullet} V = \Sym^{\bullet} \left (\superwedge^{n-1} V \oplus \superwedge^{n-1} V \oplus \dots \oplus \superwedge^1 V \right)/I^{\bullet}$$

\noindent where $I^{\bullet}$ is the two-sided ideal generated by the elements

\begin{align} \begin{split} \label{PluckRel}
(v_1 \wedge & \dots \wedge v_p) \cdot (w_1 \wedge \dots \wedge w_q) \\
& - \sum_{i=1}^p (v_1 \wedge \dots \wedge v_{i-1} \wedge w_1 \wedge v_{i+1} \wedge \dots \wedge v_p) \cdot (v_i \wedge w_2 \wedge \dots \wedge w_q)
\end{split}
\end{align}

\noindent called {\it Pl\"ucker relations}, for all $p \geq q \geq 0$. Remark that the elements \eqref{PluckRel} are the equations which define Flag varieties in general as incidence varieties inside products of Grassmann varieties. See \cite{fulton2013representation, towber1977two, towber1979young} for more details. Let us highlight that we are not going to use the natural symmetric product that $\mathbb{S}^{\bullet} V$ has as a quotient of a commutative symmetric algebra.

\noindent Eventually, since we are dealing with only one copy of $\mathbb{S}_{\lambda}V$ for any partition $\lambda$, to ease the construction of the theory we will imagine to build every Schur module using a fixed std tableau. If $\lambda \vdash d$, this one will be given by filling the diagram of $\lambda$ from top to bottom, starting from the first column, with the integers $1,\dots,d$. For instance, if $\lambda = (3,2,1)$, then the fixed tableau will be

\begin{center} \begin{ytableau} 1 & 4 & 6 \\ 2 & 5 \\ 3  \end{ytableau} . \end{center}

\smallskip

\section{Toward Schur apolarity} \label{secondasez} \setcounter{equation}{0} \medskip

In the following we will introduce the apolarity theory. For this purpose we construct an action called {\it Schur apolarity action}. \smallskip

In order to develop our theory we need to use multiplication maps $\mathcal{M}^{\lambda,\mu}_{\nu} : \mathbb{S}_{\lambda} V^* \otimes \mathbb{S}_{\mu} V^* \longrightarrow \mathbb{S}_{\nu} V^*$. Since $SL(n)$ is reductive, every representation is completely reducible and in this particular instance the behaviour of the tensor product of any two irreducible representations is well-known and computable via the Littlewood-Richardson rule which we recall. Given two Schur modules we have

\begin{equation} \label{LR} \mathbb{S}_{\lambda} V^* \otimes \mathbb{S}_{\mu} V^* \simeq \bigoplus_{\nu} N^{\lambda,\mu}_{\nu} \mathbb{S}_{\nu} V^* \end{equation}

\noindent where the coefficients $N^{\lambda,\mu}_{\nu} = N^{\mu,\lambda}_{\nu} \geq 0$ are called {\it Littlewood-Richardson coefficients}. Sufficient conditions to get $N^{\lambda,\mu}_{\nu} = 0$ are either $|\nu| \neq |\lambda| + |\mu|$ or that the diagram of either $\lambda$ or $\nu$ does not fit in the diagram of $\nu$. Therefore as soon as $N^{\lambda,\mu}_{\nu} \neq 0$, the map $\mathcal{M}^{\lambda,\mu}_{\nu} : \mathbb{S}_{\lambda} V^* \otimes \mathbb{S}_{\mu} V^* \longrightarrow \mathbb{S}_{\nu} V^*$ is non trivial and it will be a projection onto one of the addends $\mathbb{S}_{\nu} V^*$ appearing in the decomposition of $\mathbb{S}_{\lambda} V^* \otimes \mathbb{S}_{\mu} V^*$ into sum of irreducible representations as described by \eqref{LR}. Remark also that via \eqref{LR} it follows that the vector space of such equivariant morphisms has dimension $N^{\lambda,\mu}_{\nu}$. \smallskip

Before proceeding we recall some basic combinatorial facts.

\begin{definition} \label{skewtab}
\noindent Two partitions $\nu \vdash d$ and $\lambda \vdash e$, with $0 \leq e \leq d$, are such that $\lambda \subset \nu$ if $\lambda_i \leq \mu_i$ for all $i$, possibly setting some $\lambda_i$ equal to $0$. In this case $\lambda$ is a {\it subdiagram} of  $\nu$. The {\it skew Young diagram} $\nu / \lambda =(\nu_1 - \lambda_1,\dots,\nu_k - \lambda_k)$ is the diagram of $\nu$ with the diagram of $\lambda$ removed in the top left corner. A {\it skew Young tableau} $T$ of shape $\nu / \lambda$ is the diagram of $\nu / \lambda$ with a filling. The definitions of sstd and std  tableau apply also in this context.
\end{definition}

\noindent For example if $\nu = (3,2,1)$ and $\lambda = (2)$, the skew diagram $\nu / \lambda$ together with a sstd skew tableau of shape $\nu / \lambda$ is

\begin{center} \begin{ytableau} *(gray) & *(gray) & 1 \\ 1 & 2 \\ 3 \end{ytableau} . \end{center}

\noindent Skew-tableaux can be seen as a generalization of the regular tableaux by setting $\lambda = (0)$ in Definition \ref{skewtab}. 

\begin{definition} \label{skewSchurModule} Fix a std skew tableau of shape $\nu / \lambda$. One can define analogously to the formula \eqref{YoungSymm} the Young symmetrizer $c_{\nu / \lambda} : V^{\otimes |\nu|-|\lambda|} \longrightarrow V^{\otimes |\nu|-|\lambda|}$. The {\it skew Schur module} $\mathbb{S}_{\nu / \lambda} V$ is the image of $c_{\nu / \lambda}$. 
\end{definition}

Clearly also in this case two skew Schur modules determined by two different skew std tableaux of the same shape are isomorphic. Moreover they are still representations of $SL(n)$ but in general they may be reducible. Indeed it holds

\begin{equation} \label{SKEW} \mathbb{S}_{\nu / \lambda} V \cong \bigoplus_{\mu} N^{\lambda, \mu}_{\nu} \mathbb{S}_{\mu} V \end{equation}

\noindent where the coefficients $N^{\mu, \nu}_{\lambda}$ are the same appearing in \eqref{LR}. For more details see \cite[p. 83]{fulton2013representation}. Also in this case we assume that such modules are built using a std skew tableau of shape $\nu / \lambda$ filled with the integers from $1$ to $|\nu| - |\lambda|$ from top to bottom, starting from the first column.
\smallskip

\begin{definition} 
Let $\nu / \lambda$ be any skew partition and consider a sstd skew tableau $T$ of shape $\nu / \lambda$. The {\it word associated to} $T$ is the string of integers obtained from $T$ reading its entries from left to right, starting from the bottom row. The obtained word $w_1 \dots w_k$ is called either {\it Yamanouchi word} or {\it reverse lattice word} if for any $s$ from $0$ to $k-1$, the sequence $w_k w_{k-1} \dots w_{k-s}$ contains the integer $i+1$ at most many times it contains the integer $i$. For short this words will be denominated $Y${\it -words}. The {\it content} of $T$ is the sequence of integers $\mu = (\mu_1,\dots,\mu_n)$ where $\mu_i$ is the number of $i$'s in $T$. Note that this may be not a partition.
\end{definition}

\noindent For example, given

\begin{center} $T_1$ = \begin{ytableau} *(gray) & *(gray) & 1 \\ 1 & 2 \\ 3 \end{ytableau} , \quad $T_2$ = \begin{ytableau} *(gray) & *(gray) & 1 \\ 1 & 3 \\ 2 \end{ytableau} \end{center}

\noindent their associated words are $w_{T_1} = 3121$ and $w_{T_2} = 2131$ respectively. Remark that only the first one is a Y-word since in $w_{T_2}$ the subsequence $13$ has the integer $3$ appearing more times than the integer $2$. In both cases the content is $(2,1,1)$.

\begin{definition}
Let  $\lambda$ and $\nu$ be two partitions such that $\lambda \subset \nu$ and consider a skew sstd tableau $T$ of shape $\nu / \lambda$. The tableau $T$ is a {\it Littlewood-Richardson skew tableau} if its associated word is a Y-word.
\end{definition}

\begin{prop}[Prop. 3, p. 64, \cite{fulton1997young}]
Let $\mu$, $\lambda$ and $\nu$ such that $\mu, \lambda \subset \nu$ and $|\mu| + |\lambda| = |\nu|$, and consider the skew diagram $\nu / \lambda$. The number of Littlewood-Richardson skew tableau of shape $\nu / \lambda$ and content $\mu$ is exactly $N^{\lambda,\mu}_{\nu}$.
\end{prop}

\begin{remark}
The set of std tableau of shape $\lambda$ with entries in $\{1,\dots,|\lambda|\}$ is in $1$ to $1$ correspondence with the set of Y-words of length $n$ and {\it content} $\lambda$, i.e. with $\lambda_i$ times the integer $i$.  
\end{remark}

Indeed we can define two functions $\alpha$ and $\beta$

$$ \left \{ \text{std tableaux of shape}\ \lambda\ \right \} \xrightarrow{\quad\alpha\quad} \left \{ \text{Y-words of length}\ n\ \text{and content}\ \lambda \right \} $$
and 

$$ \left \{ \text{std tableaux of shape}\ \lambda\ \right \} \xleftarrow{\quad\beta\quad} \left \{ \text{Y-words of length}\ n\ \text{and content}\ \lambda \right \} $$
that are inverses each other. \smallskip

\noindent Let $T$ be a std tableau of shape $\lambda$ with entries in $\{1,\dots,|\lambda|\}$. Remark that each entry $l \in \{1,\dots,|\lambda|\}$ appears in a certain position $(i,j)$, i.e. in the box in the $i$-th row and $j$-th column. Then consider the sequence $(a_1,\dots,a_{|\lambda|})$, where we set $a_l = i$ if $l$ appears in the position $(i,j)$. Roughly, starting with the smallest entry in $\{1,\dots,|\lambda|\}$ record the row in which it appears in $T$. We define $\alpha(T)$ to be the reversed sequence 

$$\alpha(T):= \operatorname{rev}(a_1,\dots,a_{|\lambda|}) := (a_{|\lambda|},\dots,a_1),$$
where we consider $\operatorname{rev}$ as the involution that acts on the set of words reversing them. It is easy to see that $\alpha(T)$ is a Y-word since the sequences of integers are determined by the entries of $T$. \smallskip

\noindent Let $\underline{a}$ be a Y-word of content $\lambda$, so that $\lambda_i$ is the number of times in which an integer $i$ appears in $\underline{a}$. Consider its reversed sequence $\operatorname{rev}(\underline{a}) = (a_1,\dots,a_{|\lambda|})$. We define the std tableau $\beta(\underline{a})$ of shape $\lambda$ in the following way. For any integer $i \in \{1,\dots,k\}$, where $k= l(\lambda)$, consider the subsequence 

$$\operatorname{rev}(\underline{a})_i := (a_{k_1},\dots,a_{k_{\lambda_i}})$$
given by all $i$'s with respect to the order with which they appear in $\operatorname{rev}(\underline{a})$. Then we set the $(i,j)$-th entry of $\beta(\underline{a})$ to be equal to $k_j$ appearing in the subsequence $\operatorname{inv}(\underline{a})_i$. Even in this case it is clear that the image will be a std tableau of shape $\lambda$.   

\noindent For more details we refer to \cite[Definition IV.1.3, p. 252]{akin1982schur}. 

\begin{example}
Consider $\lambda = (3,1)$ and let $T$ be the tableau

$$ T = \begin{ytableau} 1 & 2 & 4 \\ 3 \end{ytableau} .$$
We apply $\alpha$ to $T$. We get at first $(a_1,\dots,a_4) = (1,1,2,1)$, so that $\alpha(T) = (1,2,1,1)$. Now we apply $\beta$ to see that actually we come back to $T$. Consider the reversed sequence $\operatorname{rev}(\alpha(T)) = (1,1,2,1)$, which has content $\lambda$. Here we have only two subsequences

$$\operatorname{rev}(\alpha(T))_1 = (a_1,a_2,a_4) = (1,1,1)\ \text{and}\ \operatorname{rev}(\alpha(T))_2 = (a_3) = (2).$$
Hence it follows that

$$ \beta((\alpha(T))) = T = \begin{ytableau} 1 & 2 & 4 \\ 3 \end{ytableau} .$$
\end{example}

\begin{definition}[Definition IV.2.2, p. 257 in \cite{akin1982schur}]
Let $T$ be a Littlewood-Richardson skew tableau of shape $\nu / \lambda$ and content $\mu$. We can define a new tableau $T'$ of shape $\nu / \lambda$ and content $\mu'$ not necessarily sstd as following. Let $\underline{a}$ be the word associated to $T$ and consider $\underline{a}' := \alpha(\beta(\underline{a})')$, where with the notation $\beta(\underline{a})'$ we mean the conjugate tableau to $\beta(\underline{a})$, i.e. the one obtained from the latter transposing it as if it would be a matrix. Then define $T'$ as the skew tableau of shape $\nu / \lambda$ whose associated word is $\underline{a}'$. 
\end{definition}

\begin{example}
Let $\nu = (3,2)$, $\lambda = (1)$ and consider

$$ T = \begin{ytableau} *(gray) & 1 & 1 \\ 1 & 2 \end{ytableau},$$
where in this case the content is $\mu = (3,1)$. The associated word is $\underline{a} = (1,2,1,1)$ which is a Y-word. Then 

$$ \beta(\underline{a}) = \begin{ytableau} 1 & 2 & 4 \\ 3 \end{ytableau}\quad \text{and}\quad \beta(\underline{a})' = \begin{ytableau} 1 & 3 \\ 2 \\ 4 \end{ytableau} $$
so that $\alpha(\beta(\underline{a})') = (3,1,2,1)$. Therefore we get

$$ T' = \begin{ytableau} *(gray) & 2 & 1 \\ 3 & 1 \end{ytableau} .$$
\end{example}

\begin{definition}[Definition IV.2.4, p. 258]
Consider three partitions $\lambda$, $\mu$ and $\nu$ such that $N^{\lambda,\mu}_{\nu} \neq 0$ and consider the diagrams of $\nu / \lambda$ and $\mu$. Let $T$ be a Littlewood-Richardson skew tableau of shape $\nu / \lambda$ and content $\mu$. When considering a Young diagram, we denote with $(i,j)$ the entry of the diagram positioned in the $i$-th row and $j$-th column. We can define a map $\sigma_T : \nu / \lambda \longrightarrow \mu$ from the entries of the diagram of $\nu / \lambda$ to the entries of the diagram of $\mu$ such that

$$\sigma_T (i,j) := ( T(i,j), T'(i,j)), $$
for every entry $(i,j)$ of $\nu / \lambda$. In the following we adopt the notation $V^{\otimes \lambda}$ to denote $V^{\otimes |\lambda|}$ in which every factor is identified with some entry $(i,j)$ of the Young diagram of $\lambda$. 
\end{definition}

\begin{remark}
Recall that we can write the Young symmetrizer $c_{\lambda}^T : V^{\otimes d} \longrightarrow V^{\otimes d}$ as the composition $b_{\lambda} \cdot a_{\lambda}$ of two endomorphisms $a_{\lambda}^T,\ b_{\lambda}^T : V^{\otimes d} \longrightarrow V^{\otimes d}$ such that on decomposable elements they act as

$$ a_{\lambda} (v_1 \otimes \dots \otimes v_d) := \sum_{\sigma \in R_{\lambda}} v_{\sigma(1)} \otimes \dots \otimes v_{\sigma(d)},$$

$$ b_{\lambda} (v_1 \otimes \dots \otimes v_d) := \sum_{\tau \in C_{\lambda}} \operatorname{sgn}(\tau) v_{\tau(1)} \otimes \dots \otimes v_{\tau(d)}.$$
In the following we are going to use these two maps. In particular remark that

$$a_{\lambda} (V^{\otimes d}) = \Sym_{\lambda} V := \Sym^{\lambda_1} V \otimes \dots \otimes \Sym^{\lambda_k} V,$$
while for $b_{\lambda}$ we have
$$b_{\lambda} (V^{\otimes d}) = \superwedge_{\lambda'} V$$
so that we get the inclusion $\mathbb{S}_{\lambda} V \subset \superwedge_{\lambda'} V$ we have already seen. 
\end{remark}

\noindent We are ready to give the definition we need of the maps $\mathcal{M}^{\lambda,\mu}_{\nu} : \mathbb{S}_{\lambda} V^* \otimes \mathbb{S}_{\mu} V^* \longrightarrow \mathbb{S}_{\nu} V^*$. 

\begin{definition} \label{defmapmult}
Let $\lambda$, $\mu$ and $\nu$ be three partitions such that $N^{\lambda,\mu}_{\nu} \neq 0$, and fix a Littlewood-Richardson tableau $T$ of shape $\nu / \lambda$ and content $\mu$. We define the map $\mathcal{M}^{\lambda,\mu}_{\nu,T} : \mathbb{S}_{\lambda} V^* \otimes \mathbb{S}_{\mu} V^* \longrightarrow \mathbb{S}_{\nu} V^*$ as the following composition of maps

\begin{align*} \mathbb{S}_{\lambda} V^* \otimes \mathbb{S}_{\mu} V^* \rightarrow &\Sym_{\lambda} V^* \otimes \Sym_{\mu} V^* \rightarrow \\
&\rightarrow \Sym_{\lambda} V^* \otimes \Sym_{\nu / \lambda} V^* \rightarrow \Sym_{\nu} V^* \rightarrow \mathbb{S}_{\nu} V^*, \end{align*}
where the first map is the transpose of the map

$$ a_{\lambda} \otimes a_{\mu} : \Sym_{\lambda} V^* \otimes \Sym_{\mu} V^* \rightarrow \mathbb{S}_{\lambda} V^* \otimes \mathbb{S}_{\mu} V^*.$$
The second map is the tensor product of the identity on $\Sym_{\lambda}V^*$ with the map

$$\Sym_{\mu} V^* \longrightarrow (V^*)^{\otimes \mu} \longrightarrow (V^*)^{\otimes \nu / \lambda} \longrightarrow \Sym_{\nu / \lambda} V^*$$
that previously embeds $\Sym_{\mu} V^*$ in $V^{\otimes |\mu|}$ denoting the factors of $V^{\otimes |\mu|}$ with the entries of the diagram of $\mu$. Then it rearranges such factors accordingly to the map $\sigma_T$, and eventually it applies the map $a_{\nu / \lambda}$ to get $\Sym_{\nu / \lambda} V^*$. 
\noindent The third map is the tensor product of the maps

$$\Sym^{\lambda_i} V^* \otimes \Sym^{\nu_i - \lambda_i} V^* \longrightarrow \Sym^{\nu_i} V^*.$$
Eventually the last map is just the skew-symmetrizing part of the Young symmetrizer $c_{\nu}$, i.e.

$$b_{\nu} : \Sym_{\nu} V^* \longrightarrow \mathbb{S}_{\nu} V^*.$$
\end{definition}

\begin{example}
Consider the partitions $\lambda = (2,1)$, $\mu = (1,1)$ and $\nu = (3,2)$, so that $N^{(2,1),(1,1)}_{(3,2)} = 1$. In particular we have only one Littlewood-Richardson skew tableau of shape $(3,2) / (2,1)$ and content $(1,1)$ which is

$$ T = \begin{ytableau} *(gray) & *(gray) & 1 \\ *(gray) & 2 \end{ytableau} . $$
We describe the image of the element 

$$ t = (v_1 \wedge v_2 \otimes v_3 - v_2 \wedge v_3 \otimes v_1) \otimes (\alpha \wedge \beta) \in \mathbb{S}_{(2,1)} V \otimes \mathbb{S}_{(1,1)} V$$
via the map

$$\mathcal{M}^{(2,1),(1,1)}_{(3,2)} : \mathbb{S}_{(2,1)} V \otimes \mathbb{S}_{(1,1)} V \longrightarrow \mathbb{S}_{(3,2)} V\ .$$
To this end we describe all the maps involved in the composition defining $\mathcal{M}^{(2,1),(1,1)}_{(3,2)}$ as described in Definition \ref{defmapmult}. Remark that the element belonging to $\mathbb{S}_{(2,1)} V$ is the one determined by the sstd Young tableau

$$ \begin{ytableau} 1 & 3 \\ 2 \end{ytableau}\ .$$
At first we have the map

$$\mathbb{S}_{(2,1)} V \otimes \mathbb{S}_{(1,1)} V \longrightarrow \Sym_{(2,1)} V \otimes \Sym_{(1,1)} V$$
that sends

$$ t\ \mapsto\ (v_1v_3 \otimes v_2 - v_2v_3 \otimes v_1 ) \otimes (\alpha \otimes \beta - \beta \otimes \alpha). $$
Then we have to apply the map $\Sym_{(2,1)} V \otimes \Sym_{(1,1)} V \longrightarrow \Sym_{(2,1)} V \otimes \Sym_{(3,2)/(2,1)} V$ to this last element. Remark that in this case the map $\sigma_T$ acts as

$$\sigma_T :\ \begin{ytableau} *(gray) & *(gray) & b \\ *(gray) & a \end{ytableau} \longrightarrow \begin{ytableau} b \\ a \end{ytableau}  $$
and moreover $\Sym_{(1,1)} V = \Sym_{(3,2)/(2,1)} V = V \otimes V$. Now we apply the map $\Sym_{(2,1)} V \otimes \Sym_{(3,2)/(2,1)} V \longrightarrow \Sym_{(3,2)} V$, which is given by the tensor product of the maps

$$\Sym^2 V \otimes \Sym^1 V \longrightarrow \Sym^3 V,\quad \Sym^1 V \otimes \Sym^1 V \longrightarrow \Sym^2 V. $$
The image of our element is 

$$ v_1v_3\alpha \otimes v_2 \beta -  v_1v_3\beta \otimes v_2 \alpha -  v_2v_3\alpha \otimes v_1 \beta + v_2v_3\beta \otimes v_1 \alpha.$$
Eventually we have to apply the map $\Sym_{(3,2)} V \longrightarrow \mathbb{S}_{(3,2)} V$ which is the skew-symmetrizing part of the Young symmetrizer. We can represent the image of such a map using Young tableaux of shape $(3,2)$

$$\begin{ytableau} 1 & 3 & \alpha \\ 2 & \beta \end{ytableau} - \begin{ytableau} 1 & 3 & \beta \\ 2 & \alpha \end{ytableau} - \begin{ytableau} 2 & 3 & \alpha \\ 1 & \beta \end{ytableau} + \begin{ytableau} 2 & 3 & \beta \\ 1 & \alpha \end{ytableau}$$
where we are considering such tableaux as elements of $V^{\otimes 5}$ as described in Section \ref{primasez}. 
\end{example}

\begin{remark}The map $\mathcal{M}^{\lambda,\mu}_{\nu,T} : \mathbb{S}_{\lambda} V \otimes \mathbb{S}_{\mu} V \longrightarrow \mathbb{S}_{\nu} V$ is clearly equivariant and by construction its image is contained in $\mathbb{S}_{\nu}V$. Moreover it is easy to see that since these maps are determined by a choice of a Littlewood-Richardson skew tableau, they all acts in different way and they are linearly independent.
\end{remark}

\medskip

\begin{remark} \label{trim} The last thing we would like to underline is that some of these multiplication maps are involved in the construction of the elements of $\mathbb{S}_{\lambda}V$ via Young symmetrizers. Indeed, for example consider $\lambda=(3,2,1)$. We know that $\mathbb{S}_{(3,2,1)} V$ is the image of the respective Young symmetrizer applied $V^{\otimes 6}$ using symmetrization along rows and then skew-symmetrizations along columns. Via this interpretation we can see  $\mathbb{S}_{\lambda}V$ obtained by adding row by row, i.e. via the following composition

$$\mathbb{S}_{(3)}V \otimes \mathbb{S}_{(2)} V \otimes \mathbb{S}_{(1)} V \longrightarrow \mathbb{S}_{(3,2)}V \otimes \mathbb{S}_{(1)} V \longrightarrow \mathbb{S}_{(3,2,1)} V $$

\noindent where the first map is $\mathcal{M}^{(3),(2)}_{(3,2)} \otimes \text{id}_{\mathbb{S}_{(1)}V}$ and the second one is $\mathcal{M}^{(3,2),(1)}_{(3,2,1)}$. Note that all these maps that add only a row are unique up to scalar multiplication since the respective Littlewood-Richardson coefficient is $1$. Hence we can see every module obtained by adding row by row in the order specified by the partition. We sometimes refer to the inverse passage, i.e. deleting row by row, as {\it trimming}.
\end{remark}

\begin{remark} \label{gradedring}
We end this section with a last remark. The space $\mathbb{S}^{\bullet} V$ together with the maps $\mathcal{M}^{\lambda,\mu}_{\nu}$ is a graded ring. Precisely the graded pieces are given by

$$\left ( \mathbb{S}^{\bullet} V \right )_a := \bigoplus_{\lambda\ :\ |\lambda| = a} \mathbb{S}_{\lambda} V. $$
Moreover, given two elements $g \in \mathbb{S}_{\lambda} V$ and $h \in \mathbb{S}_{\mu} V$ where $|\lambda| = a$ and $|\mu| = b$, then if $N^{\lambda,\mu}_{\nu} \neq 0$ we get that the element $\mathcal{M}^{\lambda,\mu}_{\nu}(g \otimes h)$ belongs to $\mathbb{S}_{\nu} V$, where $|\nu| = |\lambda| + |\mu| = a+b$. Therefore the condition

$$\left ( \mathbb{S}^{\bullet} V \right )_a \cdot \left ( \mathbb{S}^{\bullet} V \right )_b \subset \left ( \mathbb{S}^{\bullet} V \right )_{a+b}$$
is satisfied.
\end{remark}

\section{Schur apolarity action} \label{terzasez} \setcounter{equation}{0} \medskip

\noindent In this section we define the {\it Schur apolarity action} which will be the foundation of our results. \smallskip

We would like to define an apolarity action whose domain is $\mathbb{S}_{\lambda} V \otimes \mathbb{S}_{\mu} V^*$, with $\lambda$ and $\mu$ suitably chosen, which extends the symmetric and the skew-symmetric action. Naively it seems natural to require that $\mu \subset \lambda$ and that this map is 

$$ \varphi : \mathbb{S}_{\lambda} V \otimes \mathbb{S}_{\mu} V^* \longrightarrow \mathbb{S}_{\lambda / \mu} V $$

\noindent and to set $\varphi$ equal to the zero map if $\mu \not\subset \lambda$. 

Before going along with the description, we have to recall the definition of {\it skew-symmetric apolarity action} given in \cite{arrondo2021skew}.

\begin{definition} Given two integers $1 \leq h \leq k \leq \dim(V)$, it can be described as

$$ \superwedge^{k} V \otimes \superwedge^{h} V^* \longrightarrow \superwedge^{k-h} V$$ 
$$ (v_1 \wedge \dots \wedge v_k) \otimes (\alpha_1 \wedge \dots \wedge \alpha_h)\ \longmapsto \sum_{R \subset \{1,\dots,k\}} \operatorname{sign}(R) \cdot \det(\alpha_i(v_j)_{j \in R}) \cdot v_{\overline{R}} $$

\noindent where the sum runs over all the possible ordered subsets $R$ of $\{1,\dots,k\}$ of cardinality $h$ and the set $\overline{R}$ is the complementary set to $R$ in $\{1,\dots,k\}$. The element $v_{\overline{R}}$ is the wedge product of the vectors whose index is in $\overline{R}$, and $\operatorname{sign}(R)$ is the sign of the permutation which sends the sequence of integers from $1$ to $k$ to the sequence in which the elements of $R$ appear first already ordered, keeping the order of the other elements. 
\end{definition} 
\begin{remark}
The skew-symmetric apolarity action can be regarded as the composition

$$ \superwedge^h V^* \otimes \superwedge^k V \longrightarrow \superwedge^h V^* \otimes \superwedge^{h} V \otimes \superwedge^{k-h} V \longrightarrow \superwedge^{k-h}V$$

\noindent where the first map is the tensor product of the identity on the first factor and the comultiplication of the exterior algebra regarded as a bialgebra on the second one. The second map acts with the identity on the last factor and acts with the determinantal pairing $\superwedge^h V^* \otimes \superwedge^{h} V \longrightarrow \mathbb{C}$ on the first two. 
\end{remark}

\begin{definition} \label{Schurapodef} Let $\mathbb{S}_{\lambda} V \subset \superwedge_{\lambda'} V$ and $\mathbb{S}_{\mu} V^*\subset \superwedge_{\mu'} V$ be two Schur modules.  Then the Schur apolarity action is defined as the map

 $$\varphi: \mathbb{S}^{\bullet} V \otimes \mathbb{S}^{\bullet} V^* \longrightarrow \mathbb{S}^{\bullet} V$$

\noindent such that when restricted to a product $\mathbb{S}_{\lambda} V \otimes \mathbb{S}_{\mu} V^*$, it is the trivial map if $\mu \not\subset \lambda$. Otherwise it is given by the restriction of the map

$$\tilde{\varphi} : \superwedge_{\lambda'} V \otimes \superwedge_{\mu'} V^* \longrightarrow
 \superwedge_{\lambda' / \mu'} V $$

\noindent that acts as the tensor product of skew-symmetric apolarity actions $\superwedge^{\lambda_i'} V \otimes \superwedge^{\mu_i'} V^* \longrightarrow \superwedge^{\lambda_i'-\mu_i'} V$.
\end{definition}

\begin{example} \label{esschurapoflag} Consider the partitions $\lambda = (3,2,1)$ and $\mu = (2)$. We have that $\mu \subset \lambda$ so that the Schur apolarity action $\varphi$ restricted to $\mathbb{S}_{\lambda} V \otimes \mathbb{S}_{\mu} V^*$ is not trivial. Consider the element

$$ t = v_1 \wedge v_2 \wedge v_3 \otimes v_1 \wedge v_2 \otimes v_1 \in \mathbb{S}_{(3,2,1)} V$$
and let $\alpha\beta \in \Sym^2 V$ be any element. Then the Schur apolarity action acts as

\begin{align*}
&\mathcal{C}^{(3,2,1),(2)}_{t_1}(\alpha \beta) = \\
&= \sum_{i=1}^3 (-1)^{i+1} \left [ \alpha(v_i)\beta(v_1) + \alpha(v_1)\beta(v_i) \right ] \cdot v_1 \wedge \dots \wedge \hat{v_i} \wedge \dots \wedge v_3 \otimes v_2 \otimes v_1 + \\
&- \sum_{i=1}^3  (-1)^{i+1} \left [ \alpha(v_i)\beta(v_2) + \alpha(v_2)\beta(v_i) \right ] \cdot v_1 \wedge \dots \wedge \hat{v_i} \wedge \dots \wedge v_3 \otimes v_1 \otimes v_1.
\end{align*} 
\end{example}

\begin{example}
Consider $\lambda = (2,2)$ and $\mu = (1,1)$. Let
$$ t = v_1 \wedge v_2 \otimes v_1 \wedge v_3 + v_1 \wedge v_3 \otimes v_1 \wedge v_2 \in \mathbb{S}_{(2,2)}\mathbb{C}^4$$

\noindent and let $s = x_1 \wedge x_2 \in \mathbb{S}_{(1,1)}(\mathbb{C}^4)^*$. Then

\begin{align*} \varphi(t \otimes s)& = \\ & = \det \left (\begin{matrix} x_1(v_1) & x_1(v_2) \\ x_2(v_1) & x_2(v_2) \end{matrix}\right ) v_1 \wedge v_3 + \det \left (\begin{matrix} x_1(v_1) & x_1(v_3) \\ x_2(v_1) & x_2(v_3) \end{matrix}\right ) v_1 \wedge v_2  \\ & = v_1 \wedge v_3. \end{align*}
\end{example}
\medskip

\noindent A priori we are not able to say that the image is contained in the skew Schur module $\mathbb{S}_{\lambda / \mu} V$. The following fact will clear our minds.

\begin{prop} \label{imgSchurAction}
Let $ \varphi: \mathbb{S}_{\lambda} V \otimes \mathbb{S}_{\mu} V^* \to \superwedge_{\lambda'/\mu'} V$ be the Schur apolarity action restricted to $\mathbb{S}_{\lambda} V \otimes \mathbb{S}_{\mu} V^*$. Then image of $\varphi$  is contained in $\mathbb{S}_{\lambda / \mu} V$. 
\end{prop}
\begin{proof}

We prove it for two elements of the basis of $\mathbb{S}_{\lambda}V$ and $\mathbb{S}_{\mu} V^*$ given by the projections $c_{\lambda}$ and $c_{\mu}$ respectively. In the following the letters $a,b,\dots$ denote the elements of the filling of tableaux of shape $\lambda$, while greek letters $\alpha,\beta,\dots$ denote the elements of the filling of tableaux of shape $\mu$. We use the pictorial notation we have introduced when generating a basis of all these spaces in Section \ref{primasez}. Performing Schur apolarity means erasing the diagram of $\mu$ in the diagram of $\lambda$ in the top left corner contracting the respective letters coming from the elements of $V$ with those of $V^*$. Every such skew tableau comes with a coefficient given by the contraction of the $\alpha, \beta, \dots$ with the $a, b, \dots$ in a specific order according with our notation. The image of Schur apolarity between two elements of the basis is then given by a sum of skew tableaux of shape $\lambda / \mu$ with proper fillings and coefficients. Remark that skew tableaux may contain {\it disjoint} subdiagrams, i.e. diagrams which do not share neither a row or a column. For example 

 \begin{center}
\ytableausetup{nosmalltableaux}
\begin{ytableau}
*(gray)  & *(gray)  & c  \\
d & e  \\
f 
\end{ytableau} \end{center}

\noindent contains two disjoint subdiagrams, $\lambda_1 = (1)$ and $\lambda_2 = (2,1)$. Let us group the addenda in the image in such a way that they all have the same coefficients and the respective disjoint subdiagrams share the same fillings. For example, consider the tableaux

$$ U = \begin{ytableau} a & b & c \\ d & e \\ f \end{ytableau} \quad \text{and} \quad S = \begin{ytableau} \alpha & \beta \end{ytableau} $$

\noindent and perform all the permutations accordingly to the maps $c_{\lambda}$ and $c_{\mu}$. Then we have to contract every single addend of the first tableau with every addend of the second as we have said above. In all these contractions we may find a summand like

\begin{equation} \label{SchurApoImg} \alpha(a) \beta(b) \cdot \left ( \begin{ytableau} *(gray)  & *(gray)  & c  \\ d & e  \\ f \end{ytableau} - \begin{ytableau} *(gray)  & *(gray)  & c  \\ f & e  \\ d \end{ytableau} +    \begin{ytableau} *(gray)  & *(gray)  & c  \\ e & d  \\ f \end{ytableau} - \begin{ytableau} *(gray)  & *(gray)  & c  \\ f & d  \\ e \end{ytableau} \right ). \end{equation}

\noindent It is possible to find such elements for two reasons. At first, if we consider permutations of the bigger diagram which send the erased elements in the right positions, we can find disjoint subdiagrams which share the same fillings. Moreover, if the elements fixed are contracted by a single addend of the element in $\mathbb{S}_{\mu} V^*$, the coefficients are always the same. We have collected every single group of skew tableaux  such that the fillings of every subdiagram are permuted accordingly to the symmetrization rules of a std skew tableau of shape $\lambda / \mu$. Hence, every such group is the image via $c_{\lambda / \mu}$ of some element of $V^{\otimes |\lambda| - |\mu|}$. In the Example \ref{SchurApoImg}, the element is image via $c_{\lambda / \mu}$ of the tableau

$$ \alpha(a) \beta(b) \cdot \begin{ytableau} *(gray) & *(gray) & c \\ d & e \\ f \end{ytableau} . $$ 

\noindent In particular, remark that the signs due to permutations come in the right way since permutations along rows of the skew tableau come from permutations along rows of the bigger diagram. The same happens for exchanges along columns with the proper sign of the permutations. This proves that the image is contained in $\mathbb{S}_{\lambda / \mu} V$. 
\end{proof} 

\begin{remark} Proposition \ref{imgSchurAction} tells us that these action can be restricted to the known apolarity maps. In the case $\lambda = (d)$ and $\mu = (e)$ such that $e \leq d$ we have that $\mathbb{S}_{\lambda}V = \Sym^d V$ and $\mathbb{S}_{\mu} V^* = \Sym^e V^*$. It follows that $\mathbb{S}_{\lambda / \mu} V = \Sym^{d-e}V$ and by Proposition \ref{imgSchurAction} the Schur apolarity action coincides with the classic apolarity action. Transposing all the diagrams, if $\lambda = (1^d)$ and $\mu = (1^e)$, then the Schur apolarity action coincides with the skew symmetric apolarity action.
\end{remark}

\noindent We conclude with some basic definitions.

\begin{definition} \label{schurcatdef}
Fix $f \in \mathbb{S}_{\lambda}V$ and let $\mu \subset \lambda$, where $\mu$ and $\lambda$ both have length strictly less then $\dim(V)$. The map induced by the Schur apolarity action $\varphi$, introduced in Definition \ref{Schurapodef}, is

$$ \mathcal{C}^{\lambda,\mu}_f : \mathbb{S}_{\mu} V^* \longrightarrow \mathbb{S}_{\lambda / \mu} V$$
defined as $\mathcal{C}^{\lambda,\mu}_f(g):=\varphi(f \otimes g)$ for any $g \in \mathbb{S}_{\mu}V^*$, and is called {\it catalecticant map of $\lambda$ and $\mu$}, or simply {\it catalecticant map} if no specification is needed.
\end{definition}

\begin{definition}
The {\it apolar set to} $f \in \mathbb{S}_{\lambda}V$ is the vector subspace $f^{\perp}\subset\mathbb{S}^{\bullet}V^*$ defined as

$$ f^{\perp} := \bigoplus_{\mu} \ker \mathcal{C}^{\lambda,\mu}_f. $$ 
\end{definition}
The apolar set of a point $f \in \mathbb{S}_{\lambda} V$ will turn out to be useful on computing its structured rank.
\bigskip

\section{Additive rank and algebraic varieties} \label{quartasez}  \setcounter{equation}{0}\bigskip

This section is devoted to the description of rational homogeneous varieties and the use of the Schur apolarity action. We recall at first basic facts of the theory and we conclude with a result linking additive rank decompositions and the Schur apolarity.

\begin{definition} \label{Xrango}
Let $X \subset \mathbb{P}^N$ be a non degenerate algebraic variety and let $p \in \mathbb{P}^N$. The $X$-{\it rank of $p$} is
$$r_X(p) := \min \{ r\ :\ p \in \langle p_1,\dots,p_r \rangle,\ \text{with}\ p_1,\dots,p_r \in X \}. $$
\end{definition} \smallskip

\noindent Let us see in detail the rational homogeneous varieties we are looking for. Given the group $G=SL(V)$, a {\it representation of} $G$ is a vector space $W$ together with a morphism $\rho : G \longrightarrow GL(W)$. We use the notation $g \cdot w$ instead of $\rho(g) w$, with $g \in G$ and $w \in W$. If  we choose a basis for $V$, we may identify $G$ with $SL(n)$. Consider the subgroup $H$ of diagonal matrices $x = \operatorname{diag}(x_1,\dots,x_n)$. An element $w\in W$ is called {\it weight vector} with {\it weight} $\alpha = (\alpha_1,\dots,\alpha_n)$, with $\alpha_i$ integers, if

$$ x \cdot w = x_1^{\alpha_1} \dots x_n^{\alpha_n} w,\ \text{for all}\ x \in H. $$ 

\noindent Since $H$ is a subgroup in which every element commutes, the space $W$ can be decomposed as

$$ W = \bigoplus_{\alpha} W_{\alpha}, $$

\noindent where $W_{\lambda}$ is given by all weight vectors of weight $\alpha$ and is called {\it weight space}. If $W = \mathbb{S}_{\lambda} V$ every element of the basis $c_{\lambda}(w_S)$ with $S$ sstd tableau of shape $\lambda$, is a weight vector. Let $B \subset G$ be the subgroup of upper triangular matrices. A weight vector $w \in W$ is called {\it highest weight vector} if $B \cdot w = \mathbb{C}^* \cdot w$. It is well known that $W$ is an irreducible representation if and only if there is a unique highest weight vector, see \cite{fulton2013representation}. In the case of Schur modules we have

\begin{prop} \label{hwvSchur}
If $W = \mathbb{S}_{\lambda}V$, then the only highest weight vector in $W$ is $c_{\lambda}(v_U)$ up to scalar multiplication, where $U$ is the tableau of shape $\lambda$ whose $i-$th row contains only the number $i$, and $v_U \in V^{\otimes |\lambda|}$ is defined in Section \ref{primasez}.
\end{prop}  

\noindent For a proof see \cite[Lemma 4, p. 113]{fulton1997young}. The theory of highest weights and geometry are closely related. Given a highest weight vector $v$, there is a subgroup of $G$

$$ P = \{g \in G\ :\ g \cdot [v] = [v] \in \mathbb{P}(W) \} $$

\noindent called {\it parabolic subgroup}. The subgroup $P$ may not be normal and hence the quotient $G / P$ is just a space of cosets. Moreover, by the definition of $P$, the space $G/P$ can be identified with the orbit $G \cdot [v] \subset \mathbb{P}(W)$. It is a general fact that $G/P$ is compact and hence a closed subvariety of $\mathbb{P}(W)$ called {\it rational homogeneous variety}. 
\smallskip

\noindent Coming back to the symmetric and skew-symmetric cases, the highest weight vectors of the modules $\Sym^d V = \mathbb{S}_{(d)} V$ and $\superwedge^k V = \mathbb{S}_{(1,\dots,1)} V$, $k < \dim(V)$, are the elements $v_1^d$ and $v_1 \wedge \dots \wedge v_k$ respectively. The action of $G$ on these two elements generates the Veronese and the Grassmann varieties respectively. 

\begin{example} Consider $\lambda = (2,1)$. Then the respective highest weight vector is determined by the tableau

$$ U =\begin{ytableau} 1 & 1 \\ 2 \end{ytableau} $$

\noindent and $c_{\lambda}(v_U) = v_1 \wedge v_2 \otimes v_1$. The action of $G$ on such an element generates a closed orbit which may be identified with the variety

\begin{align*} \mathbb{F}(1,2;V) = \{ (V_1,V_2)\ :\ V_1 \subset V_2 \subset V,\ \dim(V_1)=1,\ &\dim(V_2)=2 \} \\
&\subset \mathbb{G}(1,V) \times \mathbb{G}(2,V) \end{align*}

\noindent called {\it flag variety} of lines in planes in $V$. 
\end{example}
In general, the minimal orbit inside $ \mathbb{P}(\mathbb{S}_{\lambda}V)$ is the following Flag variety

$$ \mathbb{F}(k_1,\dots,k_s; V) := \{ (V_1,\dots,V_s)\ :\ V_1 \subset \dots \subset V_s \subset V,\ \dim V_i = k_i \} $$

\noindent embedded with $\mathcal{O}(d_1,\dots,d_s)$, where the $k_i$ and $d_i$ are integers determined by $\lambda$ which we are going to describe in a moment. The Veronese and the Grassmann varieties appear as particular cases. Fixed a rational homogeneous variety $X \subset \mathbb{P}(\mathbb{S}_{\lambda}V)$, we will refer to the $X$-rank as $\lambda${\it -rank} in order to underline the connection with the respective representation in which $X$ is embedded. The points of $\lambda$-rank $1$ are of the form 

$$ (v_1 \wedge \dots \wedge v_{k_s})^{\otimes d_s} \otimes \dots \otimes (v_1 \wedge \dots \wedge v_{k_1})^{\otimes d_1} $$

\noindent and they represent the flag

$$ \langle v_1,\dots,v_{k_1} \rangle \subset \dots \subset \langle v_1,\dots,v_{k_s} \rangle.$$

\noindent Remark that the notation of tensors and flags may seem inverted but it is coherent with the action of the group. \smallskip
 
\begin{remark} In the classic and skew-symmetric apolarity theories, given a point of symmetric or skew-symmetric rank $1$ respectively, there is attached an ideal generated in the symmetric or exterior algebra respectively.
In analogy to the known apolarity theories, we give the following definition.
\end{remark}

\begin{definition} \label{subpoint}
Let $\lambda = (\lambda_1^{a_1},\dots,\lambda_k^{a_k})$ be a partition, where $i^j$ means that $i$ is repeated $j$ times, such that $a_1 + \dots + a_k < n$. The variety $X \subset \mathbb{P}(\mathbb{S}_{\lambda} V)$ is the flag variety $\mathbb{F}(h_1,\dots,h_k; V)$ embedded with $\mathcal{O}(d_1,\dots,d_k)$ where

$$ h_i = \sum_{j=1}^i a_i,\ \text{and}\ d_i = \lambda_i - \lambda_{i+1},\ \text{setting}\ \lambda_{k+1}=0. $$

\noindent A point $p \in X$ is of the form

$$ p = (v_1 \wedge \dots \wedge v_{h_k})^{\otimes d_k} \otimes \dots \otimes (v_1 \wedge \dots \wedge v_{h_1})^{\otimes d_1}$$

\noindent and it represents the flag 

$$W_1 = \langle v_1, \dots, v_{h_1} \rangle \subset \dots \subset W_k = \langle v_1,\dots,v_{h_k} \rangle. $$

\noindent We may assume that their annihilators are generated by

$$W_1^{\perp} = \langle x_{h_1 + 1},\dots, x_n \rangle \supset \dots \supset W_k^{\perp} = \langle x_{h_k+1},\dots,x_n \rangle. $$
Consider the spaces

$$\Sym^1 W_k^{\perp},\ \Sym^{d_k+1} W_{k-1}^{\perp},\dots,\ \Sym^{d_k+\dots+d_2+1} W_1^{\perp}$$
which we will refer to as ``generators''. We define the {\it ideal $I(p)$ associated to the point} $p$ as the ideal generated by the generators inside the graded ring $\left(\mathbb{S}^{\bullet} V^*,\ \mathcal{M}^{\lambda,\mu}_{\nu} \right )$ as described in Remark \ref{gradedring}, where $\mathcal{M}^{\lambda,\mu}_{\nu}$ are the multiplication maps introduced in Definition \ref{defmapmult}. 
\end{definition}

\begin{prop} \label{PropIdealKill} Let $p \in \mathbb{S}_{\eta} V$ be a point of $\eta$-rank $1$. Then the associated ideal $I(p)$ is such that all its elements kill $p$ via the Schur apolarity action. 
\end{prop}

\begin{proof}
Remark at first that the generators kill $p$ via the Schur apolarity action. Then one has to prove that given $g \in I(p)$ such that $\varphi(g \otimes p)=0$, then the product of $g$ and $h$ is still apolar to $p$ for any $h \in \mathbb{S}^{\bullet} V^*$. Without loss of generality we can assume that $g \in \mathbb{S}_{\lambda}V^*$ and $h \in \mathbb{S}_{\mu} V^*$. Consider a partition $\nu$ such that $N^{\lambda,\mu}_{\nu} \neq 0$. Let us denote the element $\mathcal{M}^{\lambda,\mu}_{\nu}(g \otimes h)$ only with $g \cdot h$. Clearly if $\nu \not \subset \eta$, then $g \cdot h$ kills $p$ by Definition \ref{Schurapodef}. Otherwise, to prove that $g \cdot h$ kills $p$ is enough to recall the description of the multiplication maps given in Definition \ref{defmapmult}, and to consider the diagram \medskip

\begin{tikzcd}
                                                             & \mathbb{S}_{\lambda}V^* \otimes \mathbb{S}_{\mu}V^* \otimes \mathbb{S}_{\eta} V \arrow[ld] \arrow[rd] &                                                                                 \\
\mathbb{S}_{\nu} V^* \otimes \mathbb{S}_{\eta} V \arrow[rdd] &                                                                                                       & \mathbb{S}_{\mu}V^* \otimes \mathbb{S}_{\eta / \lambda} V \arrow[d]             \\
                                                             &                                                                                                       & \mathbb{S}_{\nu / \lambda} V^* \otimes \mathbb{S}_{\eta / \lambda} V \arrow[ld] \\
                                                             & \mathbb{S}_{\eta / \nu}V                                                                              &                                                                                
\end{tikzcd}

\noindent from which it follows $\varphi(g\cdot h \otimes p)=0$ by the fact that $\varphi( h \otimes \varphi(g \otimes p)) = \varphi( h \otimes 0) =0$ for any $h \in \mathbb{S}_{\mu}V^*$ from the hypothesis.
\end{proof}

\begin{remark}
Given any element $f \in \mathbb{S}_{\eta} V$, from the proof of Proposition \ref{PropIdealKill} it follows that $f^{\perp} \subset \mathbb{S}^{\bullet} V^*$ is an ideal. Indeed, one needs only to prove the following fact. Given any element $g \in f^{\perp}$, which without loss of generality we can assume that belongs to $f^{\perp} \cap \mathbb{S}_{\lambda} V^*$ for some partition $\lambda$, and given any element $h \in \mathbb{S}_{\mu}V^*$ for a partition $\mu$, the product

$$ g \cdot h := \mathcal{M}^{\lambda,\mu}_{\nu}(g \otimes h) \in \mathbb{S}_{\nu} V^*,$$
where we assume that $N^{\lambda,\mu}_{\nu} \neq 0$ and $\nu \subset \eta$, is still an element of $f^{\perp}$, i.e. it holds

$$\varphi(g \cdot h \otimes f)=0. $$
To prove this it is enough to apply verbatim the proof of Proposition \ref{PropIdealKill}.
\end{remark}

\begin{remark}
The choice of the integers appearing on the symmetric powers of the generators depends on the embedding of the variety. For instance consider the flag variety $\mathbb{F}(1,2,3;V)$ embedded with $\mathcal{O}(1,1,1) $ in $\mathbb{S}_{(3,2,1)} V$. Let $\{v_1,\dots,v_n\}$ and $\{x_1,\dots,x_n\}$ be bases of $V$ and $V^*$ respectively, dual each other. Consider the $(3,2,1)$-rank $1$ tensor 

$$ t = v_1 \wedge v_2 \wedge v_3 \otimes v_1 \wedge v_2 \otimes v_1. $$
According with the notation of Definition \ref{subpoint} we have that

$$W_1 = \langle v_1 \rangle \subset W_2 = \langle v_1, v_2 \rangle \subset W_3 = \langle v_1,v_2,v_3 \rangle$$
and
$$W_1^{\perp} =\langle x_2,\dots,x_n \rangle \supset W_2^{\perp} = \langle x_3,\dots,x_n \rangle \supset W_3^{\perp} = \langle x_3,\dots,x_n \rangle.$$
From Definition \ref{subpoint} the generators are

$$\Sym^1 W_3^{\perp},\ \Sym^2 W_2^{\perp}\ \text{and}\ \Sym^3 W_1^{\perp}.$$
Consider the element $x_2^2x_3 \in \Sym^3 W_1^{\perp}$. It is easy to see that $\varphi(t \otimes x_2^2x_3) = 0$ since either $x_2$ or $x_3$ is always evaluated on the part of $t$ representing the subspace $V_1$ of the flag. 

\noindent On the other hand, consider the same variety embedded with $\mathcal{O}(1,2,1) $ in $\mathbb{S}_{(4,3,1)} V$. The analogous element  to $t$ is

$$s = v_1 \wedge v_2 \wedge v_3 \otimes (v_1 \wedge v_2)^{\otimes 2} \otimes v_1. $$
Remark that the linear spaces $W_i$ and $W_i^{\perp}$ are the same as before, with $i = 1,2, 3$. However it is easy to see that the element $x_2^2x_3$ is no longer apolar to $s$, indeed

$$\varphi(s \otimes x_2^2x_3) = v_1 \wedge v_2 \otimes (v_1)^{\otimes 3}.$$
Therefore, we need to change the powers appearing on the generators as described in Definition \ref{subpoint} to get as generators the spaces

$$\Sym^1 W_3^{\perp},\ \Sym^2 W_2^{\perp}\ \text{and}\ \Sym^4 W_1^{\perp}.$$
This will allow us to evaluate elements of $W_1^{\perp}$ to the part of the tensor $s$ representing $W_1$.
\end{remark}

\begin{remark}[Restriction to the known apolarities]
Let $p = l^d \in \nu_d(\mathbb{P}^{n-1})$ be a point of a Veronese variety. The point $p$ represents a line contained in $V$ and hence the respective annihilator is generated by $n-1$ linear forms. Applying Definition \ref{subpoint} we get the ideal $I(p) \subset \mathbb{S}^{\bullet}V^*$. In particular remark that the multiplication maps $\Sym^d V^* \otimes \Sym^e V^* \longrightarrow \Sym^{d+e}V^*$ are involved in the definition. Hence it is not hard to check that the intersection $I(p) \cap \Sym^{\bullet} V$ is the usual ideal of the point $p$ contained in $\Sym^{\bullet} V^*$ used in the classic apolarity theory. See \cite{iarrobino1999power} for more details.

\noindent The same happens when we consider $p = v_1 \wedge \dots \wedge v_k \in \superwedge^k V$ a point of a Grassmannian $\mathbb{G}(k,V)$. Recall that such a point represents a $k$-dimensional subspace $W$ of $V$. The annihilator $W^{\perp}$ of $W$ is $(\dim(V)-k)$-dimensional. Applying Definition \ref{subpoint} we get the ideal $I(p)$ generated inside $\mathbb{S}^{\bullet}V^*$ using the maps $\mathcal{M}^{\lambda,\mu}_{\nu}$ introduced in Definition \ref{defmapmult}. In particular remark that the multiplication maps $\superwedge^d V^* \otimes \superwedge^e V^* \longrightarrow \superwedge^{d+e} V^*$ are involved in the definition. Hence the intersection $I(p) \cap \superwedge^{\bullet} V^*$ is the usual ideal of the point $p$ contained in $\superwedge^{\bullet} V^*$ used in the skew-symmetric apolarity theory. See \cite{arrondo2021skew} for more details.
\end{remark}

\begin{example}
Let $ V = \mathbb{C}^4$ and $\lambda = (2,2)$. The minimal orbit inside $\mathbb{P}(\mathbb{S}_{(2,2)} \mathbb{C}^4)$ is the Grassmann variety $X = (\mathbb{G}(2,\mathbb{C}^4),\mathcal{O}(2))$. Let $\{v_1,\dots,v_4 \}$ be a basis of $\mathbb{C}^4$ and $\{x_1,\dots,x_4\}$ be the respective dual basis of $(\mathbb{C}^4)^*$. Let $p = (v_1 \wedge v_2)^{\otimes 2} \in \mathbb{S}_{(2,2)} \mathbb{C}^4$ be a point of $(2,2)$-rank $1$. The element associated to the sstd tableau
\begin{center} \begin{ytableau} 1 & 1 \\ 2 & 2 \end{ytableau} \end{center}
and it represents the subspace $W_1$ spanned by $v_1$ and $v_2$. Hence the annihilator is spanned by $x_3$ and $x_4$. One can readily check that 

$$ I(p) \cap \mathbb{S}_{(1)} (\mathbb{C}^4)^* = \langle x_3,x_4 \rangle,$$
$$ I(p) \cap \mathbb{S}_{(1,1)} (\mathbb{C}^4)^* = \langle x_i \wedge x_j\ :\ j \in \{3,4\} \rangle,$$
$$ I(p) \cap \mathbb{S}_{(2)} (\mathbb{C}^4)^* = \langle x_i  x_j\ :\ j \in \{3,4\} \rangle$$
using the maps $\mathbb{S}_{(1)} (\mathbb{C}^4)^* \otimes \mathbb{S}_{(1)} (\mathbb{C}^4)^* \longrightarrow \mathbb{S}_{\mu}(\mathbb{C}^4)^*$, where $\mu = (2), (1,1)$, restricted to $\Sym^1 W_1^{\perp}$. 

\noindent Consider now $\mu = (2,1)$. We have that $I(p) \cap \mathbb{S}_{(2,1)} (\mathbb{C}^4)^*$ is given as the span of the images of the maps $\mathcal{M}^{(1),(2)}_{(1,1)}$ and $\mathcal{M}^{(1),(1,1)}_{(2,1)}$ restricted to $I(p)$ in one of the factors in the domain. For instance the map $\mathcal{M}^{(1),(1,1)}_{(2,1)}$ restricted to $\langle x_3, x_4 \rangle \otimes \mathbb{S}_{(1,1)} (\mathbb{C}^4)^*$ has as image

\begin{align*} \langle (x_i \wedge \alpha \otimes \beta + \beta \wedge \alpha \otimes x_i -&x_i \wedge \beta \otimes \alpha - \alpha \wedge \beta \otimes x_i,\ \\ 
&\text{for all}\ \alpha \wedge \beta \in \mathbb{S}_{(1,1)} (\mathbb{C}^4)^*,\ i = 3,\ 4 \rangle. \end{align*}
On the other hand, if we consider the same map restricted to $\mathbb{S}_{(1)} (\mathbb{C}^4)^* \otimes (I(p) \cap \mathbb{S}_{(1,1)} (\mathbb{C}^4)^* )$, we get as image 

$$ \langle \alpha \wedge \beta \otimes \gamma - \alpha \wedge \gamma \otimes \beta,\ \text{for all}\ \alpha \in V^*,\ \beta \wedge \gamma \in I(p) \cap \mathbb{S}_{(1,1)} (\mathbb{C}^4)^*  \rangle.$$
If one consider the span of all such images together with the ones obtained from the map $\mathcal{M}^{(1),(2)}_{(2,1)}$, one gets that

$$ I(p) \cap \mathbb{S}_{(2,1)} (\mathbb{C}^4)^* = \langle c_{(2,1)}(e_S)\ :\ S\ \text{is a sstd tableau in which an entry is 3 or 4} \rangle.$$
Analogously one can compute that

$$ I(p) \cap \mathbb{S}_{(2,2)} (\mathbb{C}^4)^* = \langle c_{(2,2)}(e_S)\ :\ S\ \text{is a sstd tableau in which an entry is 3 or 4} \rangle.$$
Note that all the elements in such spaces kill $p$ via the Schur apolarity action.
\end{example}
\smallskip

\begin{example}
Let $ V = \mathbb{C}^3$ and $\lambda = (2,1)$. The minimal orbit inside $\mathbb{P}(\mathbb{S}_{(2,1)} \mathbb{C}^3)$ is the Flag variety $X = (\mathbb{F}(1,2;\mathbb{C}^3),\mathcal{O}(1,1))$. Let $\{v_1,\dots,v_3 \}$ be a basis of $\mathbb{C}^3$ and $\{x_1,\dots,x_3\}$ be the respective dual basis of $(\mathbb{C}^3)^*$. Let $p = v_1 \wedge v_2 \otimes v_1 \in \mathbb{S}_{(2,1)} \mathbb{C}^3$ be a point of $(2,1)$-rank $1$. The element associated to the sstd tableau
\begin{center} \begin{ytableau} 1 & 1 \\ 2 \end{ytableau} \end{center}
and it represents the line $W_1$ generated by $v_1$ contained in the subspace $W_2$ spanned by $v_1$ and $v_2$. Hence the annihilators are $W_1^{\perp} = \langle x_2,\ x_3 \rangle$ and $W_2^{\perp} = \langle x_3 \rangle$. In this case the generators are
$$\Sym^1 W_2^{\perp} = \langle x_3 \rangle\ \text{and}\ \Sym^2 W_1^{\perp} = \langle x_2^2, x_2x_3, x_3^2 \rangle. $$
For any $\mu \subset (2,1)$ we can check that the subspace $I(p)$ associated to $p$ is such that

$$ I(p) \cap \mathbb{S}_{(1)} (\mathbb{C}^3)^* = \langle x_3 \rangle,$$
$$ I(p) \cap \mathbb{S}_{(1,1)} (\mathbb{C}^3)^* = \langle x_1 \wedge x_3, x_2 \wedge x_3 \rangle,$$
$$ I(p) \cap \mathbb{S}_{(2)} (\mathbb{C}^3)^* = \langle x_1 x_3, x_2 x_3, x_3^2, x_2^2 \rangle,$$
$$ I(p) \cap \mathbb{S}_{(2,1)} (\mathbb{C}^3)^* = \langle c_{(2,1)}(e_S)\ :\ S\ \text{is a sstd tableau in which an entry is 3} \rangle.$$
\end{example} \smallskip

As recalled in the Introduction, in the classic, skew-symmetric respectively, apolarity theory, a result called {\it apolarity lemma}, {\it skew-symmetric apolarity lemma} respectively, links the symmetric, skew-symmetric respectively, rank of a point with the equivalent condition of the inclusion of the ideal of points of rank $1$ inside the apolar set of the point. We now approach to an analogous result called {\it lemma of Schur apolarity}, cf. Theorem \eqref{LemmaSchurApo}. At first we need a preparatory lemma.

\begin{lemma} \label{topdeg}
Let $\lambda$ be a partition of length less then $n$. Let $p \in X \subset \mathbb{P}(\mathbb{S}_{\lambda} V)$ be a point of $\lambda$-rank $1$. Then we have the equality 

$$ I(p)_{\lambda} = (p^{\perp})_{\lambda} $$

\noindent where $p^{\perp}$ is introduced in Definition \ref{schurcatdef}, and where $I(p)_{\lambda}$ and $(p^{\perp})_{\lambda}$ denotes $I(p) \cap \mathbb{S}_{\lambda} V^*$ and $(p^{\perp})\cap \mathbb{S}_{\lambda} V^*$ respectively. 
\end{lemma}

\begin{proof}
We present here a proof only for the case in which $\lambda$ is a partition such that its diagram is union of two rectangles, being the general case identical but more cumbersome in terms of notation. 

\noindent Assume that $\lambda$ is given by the union of two rectangles. This means that $\lambda = ((d+e)^k,e^{h-k})$, where $d,e >0$ and $0<k<h$. The minimal orbit $X$ is the Flag variety $(\mathbb{F}(k,h;V),\mathcal{O}(d,e))$ and a point $p \in X$ has the form

$$ p = (v_1 \wedge \dots \wedge v_{k})^{\otimes e} \otimes (v_1 \wedge \dots \wedge v_{h})^{\otimes d} $$

\noindent and we may also assume that it is the highest weight vector of the irreducible representation $\mathbb{S}_{\lambda} V$. The sstd tableau $U$ representing $p$ is

\begin{center}$U=$ \begin{ytableau} 1 & \dots & \dots & \dots & \dots & \dots & 1 \\
\vdots & \vdots & \vdots & \vdots & \vdots & \vdots & \vdots \\
k & \dots & \dots & \dots & k & \dots & k \\
\vdots & \vdots & \vdots & \vdots \\
h & \dots & \dots & h
\end{ytableau} . \end{center}

\noindent The flag represented by $p$ is

$$W_1 = \langle v_1,\dots,v_k\rangle \subset W_2 = \langle v_,\dots,v_h \rangle$$
and the respective annihilators are 

$$ W_2^{\perp} = \langle x_{h+1},\dots,x_n \rangle \subset W_1^{\perp} = \langle x_{k+1},\dots,x_n \rangle. $$
By Definition \ref{subpoint} the ``generators'' of $I(p)$ are $\Sym^1 W_2^{\perp}$ and $\Sym^{e+1} W_1^{\perp}$.

\noindent It is easy to see that $(p^{\perp})_{\lambda}$ is a hyperplane in $\mathbb{S}_{\lambda}V^*$ since the catalecticant pairing is not trivial. Moreover it is easy to see that it is generated by all those elements of the basis whose associated sstd tableau of shape $\lambda$ is different from $U$. Fix one of these elements and call $S$ the associated sstd tableau. If we prove that $c_{\lambda}(v_S)$ is in $I(p)$, then we obtain that $I(p)_{\lambda} \supset (p^{\perp})_{\lambda}$. The equality follows for dimensional reasons. 

\noindent We discriminate some cases. At first assume that $S$ contains an integer from $h+1$ to $n$, meaning that in the respective element some of the equations $x_{h+1},\dots,x_n$ appears. Remark that such an element belong to the generator $\Sym^1 W_2^{\perp}$. Assume that the such integer appears in the $(h-k) \times e$ rectangle in the bottom, precisely in the $i$-th row of such rectangle. Call $R$ such row. Using the map

$$ \mathcal{M}^{\lambda_1,\lambda_2}_{\lambda} : \mathbb{S}_{\lambda_1} V^* \otimes \mathbb{S}_{\lambda_2} V^* \longrightarrow \mathbb{S}_{\lambda} V^*$$
where $\lambda_1 = ((d+e)^k,e^i)$ and $\lambda_2 = (e^{h-k-i})$, we can see $c_{\lambda}(v_S)$ as image of the product $c_{\lambda_1}(v_{S_1}) \otimes c_{\lambda_2}(v_{S_2})$, where $S_1$ and $S_2$ are the sstd tableaux of shape $\lambda_1$ and $\lambda_2$ respectively obtained by separating the sstd tableau $S$ at $R$. We claim that such $c_{\lambda_1}(v_{S_1})$ belongs to $I(p)$. To see this it is enough to consider the map

$$ \mathcal{M}^{\lambda_{1,1},(e)}_{\lambda_1} : \mathbb{S}_{\lambda_{1,1}} V^* \otimes \mathbb{S}_{(e)} V^* \longrightarrow \mathbb{S}_{\lambda_1} V^* $$
where $\lambda_{1,1}$ is $\lambda_1$ with the bottom line removed, since the element $c_{\lambda_1} (v_{S_1})$ can be obtained as the image of $c_{\lambda_{1,1}}(v_{S_{1,1}}) \otimes c_{(e)} (v_R)$. The claim follows since we can easily see that $c_{(e)}(v_R) \in I(p)$. Pictorially the last map acts as

$$ \ydiagram{6,6,6,3,3} \otimes \ydiagram{3} \longrightarrow 
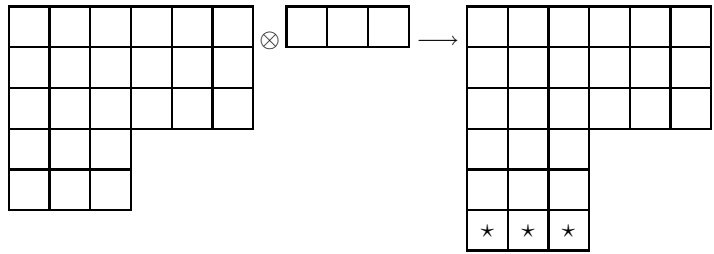
\begin{ytableau} *(white) & *(white) & *(white) & *(white) & *(white) & *(white) \\ 
*(white) & *(white) & *(white) & *(white) & *(white) & *(white) \\
*(white) & *(white) & *(white) & *(white) & *(white) & *(white) \\
*(white) & *(white) & *(white)  \\
*(white) & *(white) & *(white)  \\
\star & \star & \star 
\end{ytableau}$$
The case in which the integer from $h+1,\dots,n$ appears in upper $k \times (d+e)$ rectangle can be treated in the same way. Indeed, it is enough ``remove '' the bottom $(h-k) \times e$ rectangle and then a similar argument applies.

\noindent Suppose now that none of the integers $h+1,\dots,n$ appears in $S$. This means that the $h \times e$ rectangle  of $S$ on the left is equal to the one of $U$. On the other hand, the $k \times d$ rectangle on the right must differ from the same rectangle contained in $U$. Hence at least one integer from $k+1$ to $h$ appears in the $k \times d$ rectangle on the right. Observe now that following the construction of this element via the Young symmetrizer, some symmetrizations along rows are prescribed. Indeed note that permutations that fix the rectangle $k \times d$ on the right, fix also the other rectangle since this one has the $i$-th row filled with $i$'s. Moreover, if a permutation exchange an element from the rectangle on the right with the one on the left we get only two situations. If we exchange two equal numbers, we get the same picture we started with; otherwise after skew-symmetrizing along columns we get the zero element. This happens since we will get two equal numbers in the same column due to the fact that all the integers of the second rectangle are contained in the first one. For example consider the sstd tableau

\begin{center} $S=$\begin{ytableau} *(gray) 1 & *(gray) 1 & 1 & 2 \\ *(gray) 2 & *(gray) 2 & 2 & 3 \\ *(gray) 3 & *(gray) 3 \\ *(gray) 4 & *(gray) 4  
\end{ytableau} , \end{center} 

\noindent where we have coloured in different ways the rectangles we have referred to so far. Then whenever we exchange integers along rows between the two rectangles we obtain either a diagram with an integer repeated at least twice in a column, or we get that the first rectangle is the same. Then using the Young symmetrizer, the first one goes to zero. This allows us to perform the following construction. Select an integer that appears in the $k \times d$ rectangle on the right of $S$ which is not in the $k \times d$ rectangle of $U$. For instance if $S$ is as above, in that case $U$ is

\begin{center} $U=$\begin{ytableau} 1 & 1 & 1 & 1 \\ 2 & 2 & 2 & 2 \\ 3 & 3 \\ 4 & 4  
\end{ytableau}\ , \end{center} 

\noindent and we may consider the number $3$ since it appears in the $2 \times 2$ rectangle on the right of $S$ but not in the same rectangle of $U$. Coming back to $S$, we can see the element $c_{\lambda}(v_S)$ as image of a series of products. At first we have to trim the rows of $S$ which do not contain the chosen integer in the $h-k$ bottom rows of the $h \times e$ rectangle on the left. For instance in our case we get that $c_{\lambda}(v_S)= \mathcal{M}(c_{(4,4,2)} (v_{S_1}) \otimes c_{(2)} (v_{S_2}))$, where $S_1$ and $S_2$ are the two tableaux

\begin{center}$S_1 =$ \begin{ytableau} 1 & 1 & 1 & 2 \\ 2 & 2 & 2 & 3 \\ 3 & 3 
\end{ytableau}\ , $\qquad S_2=$ \begin{ytableau} 4 & 4 \end{ytableau}\ . \end{center} 
We can now notice that $c_{(4,4,2)}(v_S)$ is the image $\mathcal{M}^{(4^2),(2)}_{(4,4,2)}(c_{(4^2)}(v_{S_3}) \otimes c_{(2)}(v_{S_4}))$, where $S_3$ and $S_4$ are the tableaux

\begin{center} $S_3= $ \begin{ytableau} 1 & 1 & 1 & 2 \\ 2 & 3 & 3 & 3 
\end{ytableau}\ ,  $\qquad S_4=$ \begin{ytableau} 2 & 2 \end{ytableau}\ .\end{center}
Now proceeding as in the rectangular case on can see that the element $c_{(4^2)}(v_{S_3})$ belong to $I(p)$. Indeed it is given by the product $-\mathcal{M}^{(4),(4)}_{(4^2)}(c_{(4)}(v_{S_5}) \otimes c_{(4)}(v_{S_6}))$, where $S_5$ and $S_6$ are the tableaux

\begin{center} $S_5 =$ \begin{ytableau}   2 & 3 & 3 & 3 
\end{ytableau}\ , $\qquad S_6 = $ \begin{ytableau} 1 & 1 & 1 & 2 \end{ytableau} \end{center}
One can see that $c_{(4)}(v_{S_3})$ is an element of $I(p)$ by Definition \ref{subpoint} and we can conclude that $c_{\lambda}(v_S)$ belongs to $I(p)$.

\noindent The general case, i.e. when more than $2$ subspaces are involved follows exactly as the cases discussed above. Also in this case we may choose $p$ as the highest weight vector of the representation, whose associated sstd tableau of shape $\lambda$ is $U(\lambda)$, and $(p^{\perp})_{\lambda}$ is generated by all $c_{\lambda}(v_S)$ where $S$ is a sstd tableau of shape $\lambda$ different from $U$. Then one proves the inclusion $I(p)_{\lambda} \supset (p^{\perp})_{\lambda}$ and the equality will follow for dimensional reasons. Eventually one has to discuss several cases divided on which integers appear in $S$. This concludes the proof.
\end{proof}

\noindent We are now ready to state the main result of the work.

\begin{theo}[Lemma of Schur apolarity] \label{LemmaSchurApo}
Let $\lambda$ be a partition of length less than $n$ and let $\mathbb{S}_{\lambda}V$ be an irreducible representation of $SL(n)$ together with the respective minimal orbit $X\subset\mathbb{P}(\mathbb{S}_{\lambda}V)$. Let also $p_1,\dots,p_r \in X$. Then the following are equivalent:
\begin{enumerate}
\item[$(1)$] we have that $f \in \langle p_i,\ 1 \leq i \leq r\rangle$,
\item[$(2)$] we have the inclusion $I(p_1,\dots,p_r) \subset f^{\perp}$, where $I(p_1,\dots,p_r)$ is the ideal $\bigcap_{i=1}^r I(p_i)$.
\end{enumerate}
\end{theo}

\begin{proof}
We prove the two implications separately. Assume at first that $f = c_1 p_1 + \dots + c_r p_r$. Since the elements of the subspace $I(p_i)$ kill $p_i$ via Schur apolarity action for every $p_i$, we get that every element $I(p_i,\dots,p_r) = \bigcap_{i=1}^r I(p_i)$ kills every element in $\langle p_1,\dots,p_r \rangle$, and hence kills $f$.

\noindent Assume now that the inclusion $I(p_1,\dots,p_r) \subset f^{\perp}$ holds. Using the notation of Lemma \ref{topdeg}, this clearly implies that $I(p_1,\dots,p_r)_{\lambda} \subset (f^{\perp})_{\lambda}$. By Lemma \ref{topdeg} this means that
\begin{equation} \label{lemapo} I(p_1,\dots,p_r)_{\lambda} = \bigcap_{i=1}^r I(p_i)_{\lambda} = \bigcap_{i=1}^r (p_i^{\perp})_{\lambda} \subset (f^{\perp})_{\lambda}. \end{equation}
We may consider then the restricted Schur apolarity action
$$ \varphi : \mathbb{S}_{\lambda} V \otimes \mathbb{S}_{\lambda} V^* \longrightarrow \mathbb{C} $$
which is clearly a perfect pairing. Moreover the induced catalecticant map $\mathcal{C}^{\lambda,\lambda}_g$ is such that $\ker \mathcal{C}^{\lambda,\lambda}_g = (g^{\perp})_{\lambda}$, for any $g \in \mathbb{S}_{\lambda} V$. Hence we may interpret the hyperplane $(g^{\perp})_{\lambda}$ as the set of linear forms which vanish on the line $\langle g \rangle$. Hence applying the orthogonal to \eqref{lemapo}, i.e. on the both sides of

$$ \bigcap_{i=1}^r (p_i^{\perp})_{\lambda} \subset (f^{\perp})_{\lambda} $$
and uusing that the pairing is non-degenerate, we get that $f \in \langle p_1,\dots,p_r \rangle$. This concludes the proof.
\end{proof} \medskip

\begin{example}
Consider the complete flag variety $\mathbb{F}(1,2,3; \mathbb{C}^4)$ embedded with $\mathcal{O}(1,1,1)$ in $\mathbb{P} (\mathbb{S}_{(3,2,1)} \mathbb{C}^4)$ and consider the element

$$ t = v_1 \wedge v_2 \wedge v_3 \otimes v_1 \wedge v_2 \otimes v_3 - v_1 \wedge v_2 \wedge v_3 \otimes v_2 \wedge v_3 \otimes v_1 \in \mathbb{S}_{(3,2,1)} \mathbb{C}^4. $$
We would like to compute the $\lambda = (3,2,1)$-rank of $t$. We can easily exclude that $t$ has not rank $1$. Indeed it is easy to see that for any tensor $p$ of $\lambda$-rank $1$, the rank of the catalecticant map $\mathcal{C}^{(3,2,1),(2)}_p$ must be $5$. Since the rank of $\mathcal{C}^{(3,2,1),(2)}_t$ is $6$ we can already state that $t$ has not $\lambda$-rank $1$.  

\noindent Consider the two points $t_1$ and $t_2$ of $\lambda$-rank $1$ that represents the flags

$$\langle v_1+v_3 \rangle \subset \langle v_2,v_1+v_3 \rangle \subset \langle v_2,v_3,v_1+v_3 \rangle$$
$$\langle v_1-v_3 \rangle \subset \langle v_2,v_1-v_3 \rangle \subset \langle v_2,v_3,v_1-v_3 \rangle.$$
Let us call the respective subspaces $(V_1,V_2,V_3)$ and $(W_1,W_2,W_3)$. The respective orthogonal spaces are

$$ \langle x_4 \rangle \subset \langle x_4,x_1-x_3 \rangle \subset \langle x_4,x_1-x_3,x_2 \rangle $$
$$ \langle x_4 \rangle \subset \langle x_4,x_1+x_3 \rangle \subset \langle x_4,x_1+x_3,x_2 \rangle. $$
We call them $(V_3^{\perp},V_2^{\perp},V_1^{\perp})$ and $(W_3^{\perp},W_2^{\perp},W_1^{\perp})$ respectively. We can note that the following inclusions holds 

$$V_3^{\perp} \cap W_3^{\perp} = V_3^{\perp} = W_3^{\perp} \subset \ker \mathcal{C}^{(3,2,1),(1)}_t,$$ 
$$ \Sym^{2} V_2^{\perp}\cap \Sym^2 W_2^{\perp} = \langle x_4^2 \rangle \subset \ker \mathcal{C}^{(3,2,1),(2)}_t, $$
$$ \Sym^{3} V_1^{\perp}\cap \Sym^3 W_1^{\perp} = \langle x_4^3, x_2x_4^2,x_2^2x_4,x_2^3 \rangle \subset \ker \mathcal{C}^{(3,2,1),(3)}_t. $$
Moreover one can see after some computations that also the inclusion $I(p_1,p_2) \subset t^{\perp}$ holds. Hence by Theorem \ref{LemmaSchurApo} we get that $t$ can be written as a linear combination of $t_1$ and $t_2$. Solving a simple linear system we get that 

\begin{align*}
t = \frac{1}{2} (v_1 + v_3) \wedge v_2 \wedge v_3 &\otimes (v_1 + v_3) \wedge v_2 \otimes (v_1 + v_3) + \\
&- \frac{1}{2} (v_1 - v_3) \wedge v_2 \wedge v_3 \otimes (v_1 - v_3) \wedge v_2 \otimes (v_1 - v_3).
\end{align*}

\end{example}

\section{Catalecticants and $\lambda$-rank} \label{quintasez} \setcounter{equation}{0} \bigskip

This section is devoted to exploit the link between $\lambda$-rank $1$ tensors and the rank of catalecticant maps. In particular we will see that having $\lambda$-rank $1$ will be equivalent to specific ranks of the catalecticant maps. A disclaimer has to be made. Of course one can use the equations of these varieties to test if the point has $\lambda$-rank $1$. However, the goal of this study is to get a prediction on the $\lambda$-rank of a tensor in terms of the rank of specific catalecticant maps if possible. \smallskip

Firstly, let us consider the case $X = (\mathbb{G}(k,V),\mathcal{O}(d))$ embedded in $\mathbb{P}(\mathbb{S}_{(d^k)}V)$. Recall that the points of $X$ represent $k$-dimensional linear subspaces of $V$ counted $d$ times, i.e. something like

$$ (v_1 \wedge \dots \wedge v_k)^{\otimes d}  \in X $$
for some $v_1,\dots,v_k \in V$. Fix a point $p \in X$. In particular we may choose the highest weight vector of the representation so that we get also a description of $p$ in terms of a std Young tableau. For any $\mu \subset (d^k)$ we get a catalecticant map 
$$\mathcal{C}^{(d^k),\mu}_p : \mathbb{S}_{\mu} V^* \longrightarrow \mathbb{S}_{(d^k) / \mu} V. $$
If $d$ and $k$ are sufficiently large we get plenty of these maps and an analysis of all of them is not recommended. Hence we start with a couple of examples to make the situation clear and then we move to the general case.
\smallskip

Consider the variety $X = (\mathbb{G}(2,\mathbb{C}^4),\mathcal{O}(2))$ embedded in $\mathbb{P}(\mathbb{S}_{(2,2)}\mathbb{C}^4) \simeq \mathbb{P}^{19}$. Take $p \in X$ as the highest weight vector of $\mathbb{S}_{(2,2)}\mathbb{C}^4$. The partitions $\mu \subset (2,2)$ are only $\mu = (1)$, $(1,1)$, $(2)$, $(2,1)$ and $(2,2)$. We describe some of the related catalecticant maps. We start with $\mu = (1)$.

\begin{prop}
Let $p \in \mathbb{P}(\mathbb{S}_{(2,2)}\mathbb{C}^4)$. Then $p$ has $(2,2)$-rank $1$ if and only if $\operatorname{rk}(\mathcal{C}^{(2,2),(1)}_p) = 2$. 
\end{prop}

\begin{proof}
Let $p = (v_1 \wedge v_2)^{\otimes 2} \in X$. Then 
$$ \mathcal{C}^{(2,2),(1)}_p (x) = x(v_1)  v_2 \otimes v_1 \wedge v_2 - x(v_2) v_1 \otimes v_1 \wedge v_2 \in \mathbb{S}_{(2,2)/(1)} \mathbb{C}^4$$
and it is easy to see that $\ker \mathcal{C}^{(2,2),(1)}_p = \langle x_3,x_4 \rangle$. Hence the rank of the map is $2$. 

\noindent On the other hand, assume that $\operatorname{rk}(\mathcal{C}^{(2,2),(1)}_p) = 2$. This means that in the kernel of this catalecticant map there are two independent linear forms which vanish on $p$ via Schur apolarity action. Assume that $p$ is a sum of $(2,2)$-rank $1$ tensors $p_1,\dots,p_r$. We may assume that $p_i = (v_1^i \wedge v_2^i)^{\otimes 2}$, representing the subspace $\langle v_1^i,v_2^i \rangle$, with $i =1,2$. If $p_i$ and $p_j$ are independent as elements of the ambient space, the related vector subspaces must be different. Since this catalecticant actually evaluates linear forms on these subspaces, if two points or more are independent we get that strictly less than $2$ linear forms vanish on $p$. This cannot happen since by hypothesis we have $\operatorname{rk}(\mathcal{C}^{(2,2),(1)}_p) = 2$. Hence the only possibility is that all the $p_i'$s are multiple each other and hence $p = p_1$ which is what we wanted.
\end{proof}

\begin{remark}
If $p \in X$, then the generating elements of $\ker \mathcal{C}^{(2,2),(1)}_p$ are exactly the equations cutting out the subspace associated to $p$. Moreover, from this we get that for any element $t \in \mathbb{S}_{(2,2)} \mathbb{C}^4$, if $t = t_1 +\dots + t_r$ then

$$ \operatorname{rk} \mathcal{C}^{(2,2),(1)}_{t} = \operatorname{rk} \mathcal{C}^{(2,2),(1)}_{t_1 + \dots + t_r} \leq \sum_{i=1}^r \operatorname{rk} \mathcal{C}^{(2,2),(1)}_{t_i} = 2r$$
from which we get that given any $t$, if $\operatorname{rk} \mathcal{C}^{(2,2),(1)}_t = r$, then $t$ has $(2,2)$-rank at least $\lceil \frac{r}{2} \rceil$. Unfortunately it is evident that this map can give information up to $(2,2)$-rank $2$. 
\end{remark}

\noindent Let us see now the case of $\mu = (1,1)$. 

\begin{prop} \label{rango1G24}
Let $p \in \mathbb{P}(\mathbb{S}_{(2,2)}\mathbb{C}^4)$. Then $p$ has $(2,2)$-rank $1$ if and only if $\operatorname{rk}(\mathcal{C}^{(2,2),(1,1)}_p) = 1$. 
\end{prop}

\begin{proof}
Let $p = (v_1 \wedge v_2)^{\otimes 2} \in X$. Then 
$$ \mathcal{C}^{(2,2),(1,1)}_p(x_1 \wedge x_2) = \det \left ( \begin{matrix} x_1(v_1) & x_1(v_2) \\ x_2(v_1) & x_2(v_2)
\end{matrix} \right ) \cdot v_1 \wedge v_2$$
and it is easy to see that $\ker \mathcal{C}^{(2,2),(1,1)}_p = \langle x_i \wedge x_j: i<j,\ j = 3,4 \rangle$. Hence the rank of the map is $1$.

\noindent On the other hand assume that $\operatorname{rk}(\mathcal{C}^{(2,2),(1,1)}_p) = 1$. Suppose again that $p$ is a sum of $(2,2)$-rank $1$ elements $p_1,\dots,p_r$, with $p_i = (v_{1,i} \wedge v_{2,i})^{\otimes 2}$ with $i=1,\dots,r$.
Suppose by contradiction that $p$ has $(2,2)$-rank $2$ or more. Then we can find at least two points $p_i$ and $p_j$ which obviously are linearly independent. This implies that $v_{1,i} \wedge v_{2,i}$ and $v_{1,i} \wedge v_{2,i}$ are also linearly independent. On the other hand the image of the catalecticant map must be spanned by at least these two elements and hence its rank is at least $2$. This is a contradiction with the hypothesis on the rank of the catalecticant map. Hence $p$ must have $(2,2)$-rank $1$ and this concludes the proof.
\end{proof}

\begin{example}
Let 
$$t = (v_1 \wedge v_2)^{\otimes 2} + (v_1 \wedge v_3)^{\otimes 2} + (v_1 \wedge v_4)^{\otimes 2} + (v_2 \wedge v_3)^{\otimes 2} + (v_2 \wedge v_4)^{\otimes 2} + (v_3 \wedge v_4)^{\otimes 2}$$
be an element of $\mathbb{S}_{(2,2)}\mathbb{C}^4$. One can easily compute that $\mathcal{C}^{(2,2),(1,1)}_t$ is the $6 \times 6$ identity matrix whose rank is $6$. Hence by Proposition \ref{rango1G24} we can say that $t$ has $(2,2)$-rank at least $6$. Since in the above formula $t$ is written as a sum of $6$ elements of rank $1$, we can already conclude that the decomposition is minimal and $t$ has $(2,2)$-rank $6$.
\end{example}

We omit the study of the other cases $\mu = (2),(2,1)$ since they can be treated with analogous arguments. The case $\mu = (2,2)$ is trivial since the catalecticant matrix is a vector and has always rank $1$ if computed on non zero elements. In Table \ref{tabellaGrass} we give the collection of this study applied to all the $\mu \subset (2,2)$. The cases $(1)$ and $(2,1)$ have been put together since the respective catalecticant maps are transpose each other.

\begin{table}[ht] \begin{center}
\begin{tabular}{| c | c | c |} 
\hline $\mu$ & $\operatorname{rk}(\mathcal{C}^{(2,2),\mu}_t)$ & order of the matrix \\
\hline
$(1)$, $(2,1)$ & $2$ & $4 \times 20$ / $20 \times 4$ \\
\hline
$(2)$ & $3$ & $10 \times 10$ \\ 
\hline $(1,1)$ & 1 & $6 \times 6$ \\
\hline
\end{tabular}
\caption{Ranks of the catalecticants on $(2,2)$-rank $1$ elements.} \label{tabellaGrass}  \end{center}
\end{table}

Looking at the Table \ref{tabellaGrass} it is evident that the catalecticant map which gives more information about the $(2,2)$-rank is the one determined by $\mu = (1,1)$. Note that roughly this is the catalecticant map which chops a column in order to split the diagram of $(2,2)$ in two equal halves. 

\smallskip
We analyse now the case of any Grassmann variety embedded with $\mathcal{O}(d)$. We may assume that $\lambda = (d^k)$ with any $d$ and $k$ positive integers such that $k < \dim(V)$. In this case the diagram of $\lambda$ is a $k \times d$ rectangle. Following the results of the previous study, we can check if the catalecticant maps which erase a certain number of columns give good bounds on the $(d^k)$-rank of a tensor. As previously told, since we cannot check the rank of every existing catalecticant map, we cannot say that the bound determined by the maps we are going to compute is the best one. 

Let $p \in X$ be a point of $(d^k)$-rank $1$. We can take $p$ as the highest weight of the representation and write 

$$ p = (v_1 \wedge \dots \wedge v_k)^{\otimes d}.$$
Let $h \leq d$ and consider the partition $(h^k)$. Pictorially it is the $k \times h$ rectangle contained in the left upper corner of the $k \times d$ rectangle representing $\lambda$. We can consider the map

$$\mathcal{C}^{(d^k),(h^k)}_p : \mathbb{S}_{(h^k)} V^* \longrightarrow \mathbb{S}_{(d^k)/(h^k)} V. $$
It is easy to prove that $\mathbb{S}_{(d^k)/(h^k)} V \simeq \mathbb{S}_{(d-h)^k} V$. Indeed remark that the only possible Y-word that can be written inside the skew diagram of $(d^k)/(h^k)$ is the one of content $(d-h)^k$. Remark also that 
$$\mathcal{C}^{(d^k),(h^k)}_p=\left(\mathcal{C}^{(d^k),((d-h)^k)}_p\right)^t$$
and hence it makes sense to focus only to the case $h \leq \lfloor \frac{d}{2} \rfloor$. 

\begin{prop} \label{rango1Gqualunque}
Let $p \in \mathbb{P}(\mathbb{S}_{(d^k)}V)$. Then $p$ has $(d^k)$-rank $1$ if and only if $\operatorname{rk}(\mathcal{C}^{(d^k),(h^k))}_p) = 1$ for any $h \leq \lfloor \frac{d}{2} \rfloor$. 
\end{prop}

\begin{proof}
Assume that $p$ has $(d^k)$-rank $1$. We may write 

$$ p = (v_1 \wedge \dots \wedge v_k)^{\otimes d}.$$
It is evident that 

$$ \operatorname{im}\left(\mathcal{C}^{(d^k),(h^k)}_p \right) = \langle (v_1 \wedge \dots \wedge v_k)^{\otimes d-h} \rangle $$
and hence the rank of the map is $1$.

\noindent Suppose now that $\operatorname{rk}(\mathcal{C}^{(d^k),(h^k)}_p )= 1$ and that $p$ can be written as sum of $p_1,\dots,p_r$ of $(d^k)$-rank $1$. We may assume that $p_i = (v_{1,i}\wedge \dots \wedge v_{k,i})^{\otimes d}$. If $p$ has $(d^k)$-rank $2$ or more, then there are at least two points $p_i$ and $p_j$ such that $v_{1,i}\wedge \dots \wedge v_{k,i}$ and $v_{1,j}\wedge \dots \wedge v_{k,j}$ are linearly independent. In particular we will have that $(v_{1,i}\wedge \dots \wedge v_{k,i})^{\otimes d-h}$ and $(v_{1,j}\wedge \dots \wedge v_{k,j})^{\otimes d-h}$ are contained in the image of the catalecticant map and hence its rank will be at least $2$. This is in contradiction with our hypothesis and hence $p$ must have $(d^k)$-rank $1$. This concludes the proof.
\end{proof}

\begin{remark}
Given a tensor $t \in \mathbb{S}_{(d^k)} \mathbb{C}^n$, by Proposition \ref{rango1Gqualunque} we get that if the rank of $\mathcal{C}^{(d^k),(h^k)}_t$ is equal to $r$, then $t$ has $(d^k)$-rank at least $r$. Since by a simple count 

$$\dim \mathbb{S}_{(h^k)} \mathbb{C}^n = \left( \prod_{j=k}^{n-1} \frac{h+j}{j} \right) \left( \prod_{j=k-2}^{n-2} \frac{h+j}{j} \right ) \dots \left ( \prod_{j=1}^{n-k} \frac{h+j}{j} \right )$$
we have that $\dim \mathbb{S}_{(a^k)} \mathbb{C}^n > \mathbb{S}_{(b^k)} \mathbb{C}^n$ if $a > b$. Hence the most square catalecticant map is the one given by $h = \lceil \frac{d}{2} \rceil$. 
\end{remark} \smallskip

We discuss now the case of any flag variety. The situation here is a bit different as the following example shows.

\begin{example} \label{EsempioFakeRk2}
Consider the complete flag variety $\mathbb{F}(1,2,3;\mathbb{C}^4)$ embedded with $\mathcal{O}(1,1,1)$ in $\mathbb{P}(\mathbb{S}_{(3,2,1)} \mathbb{C}^4)$. Consider the element $ t$

\begin{equation} \label{FakeRk2} t = v_1 \wedge v_2 \wedge v_3 \otimes v_1 \wedge v_2 \otimes v_1 + v_1 \wedge v_2 \wedge v_3 \otimes v_2 \wedge v_3 \otimes v_3 \in \mathbb{P}(\mathbb{S}_{(3,2,1)} \mathbb{C}^4). \end{equation}
By \eqref{FakeRk2} the element is written as a sum of two $(3,2,1)$-rank $1$ elements representing the flags 

$$\langle v_1 \rangle \subset \langle v_1, v_2 \rangle \subset \langle v_1,v_2,v_3 \rangle, \quad \langle v_3 \rangle \subset \langle v_2, v_3 \rangle \subset \langle v_1,v_2,v_3 \rangle.$$
Hence the $\lambda$-rank of $t$ is at most $2$. Consider for a moment only the first addend

$$ t_1 = v_1 \wedge v_2 \wedge v_3 \otimes v_1 \wedge v_2 \otimes v_1 $$
and consider the catalecticant map $\mathcal{C}^{(3,2,1),(2)}_{t_1} : \mathbb{S}_{(2)} (\mathbb{C}^4)^* \longrightarrow \mathbb{S}_{(3,2,1)/(2)} \mathbb{C}^4$. From Example \ref{esschurapoflag} We have that

\begin{align*}
&\mathcal{C}^{(3,2,1),(2)}_{t_1}(\alpha \beta) = \\
&= \sum_{i=1}^3 (-1)^{i+1} \left [ \alpha(v_i)\beta(v_1) + \alpha(v_1)\beta(v_i) \right ] \cdot v_1 \wedge \dots \wedge \hat{v_i} \wedge \dots \wedge v_3 \otimes v_2 \otimes v_1 + \\
&- \sum_{i=1}^3  (-1)^{i+1} \left [ \alpha(v_i)\beta(v_2) + \alpha(v_2)\beta(v_i) \right ] \cdot v_1 \wedge \dots \wedge \hat{v_i} \wedge \dots \wedge v_3 \otimes v_1 \otimes v_1.
\end{align*}
It is easy to see that

$$\ker \mathcal{C}^{(3,2,1),(2)}_{t_1} = \langle x_1x_4, x_2x_4,x_3x_4,x_4^2,x_3^2 \rangle$$
and hence $ \operatorname{rk}(\mathcal{C}^{(3,2,1),(2)}_{t_1}) = 5$. Clearly this equality holds for any $t_1 \in \mathbb{F}(1,2,3;\mathbb{C}^4)$. Coming back to $t$, if we consider the same catalecticant map we get

$$\ker \mathcal{C}^{(3,2,1),(2)}_t = \langle x_1x_4, x_2x_4,x_3x_4,x_4^2 \rangle$$
which means that $ \operatorname{rk}(\mathcal{C}^{(3,2,1),(2)}_{t_1}) = 6$. This implies that $t$ cannot have $(3,2,1)$-rank $1$ and hence the decomposition \eqref{FakeRk2} is minimal. 

\noindent Differently to the case of Grassmann varieties, chopping a column of the diagram of $(3,2,1)$ does not give the right information about the $(3,2,1)$-rank $1$ tensors. Indeed if we consider the partition $(1,1,1) \subset (3,2,1)$ together with the respective catalecticant map we have that

$$ \ker \mathcal{C}^{(3,2,1),(1,1,1)}_t = \langle x_1 \wedge x_2 \wedge x_4,  x_1 \wedge x_3\wedge x_4,  x_2 \wedge x_3 \wedge x_4 \rangle$$
and hence $\operatorname{rk}(\mathcal{C}^{(3,2,1),(1,1,1)}_{t}) = 1$ which happens also for any $t \in \mathbb{F}(1,2,3;\mathbb{C}^4)$. This is because the two flags appearing in the decomposition share the same biggest subspace $\langle v_1,v_2,v_3 \rangle$.
\end{example}

Even though it seems that there is no connection between the rank of catalecticant maps and the $\lambda$-rank of tensors, we can give a lower bound on the $\lambda$-rank of a tensor. Let $X= \left(\mathbb{F}(n_1,\dots,n_s;\mathbb{C}^n),\mathcal{O}(d_1,\dots,d_s)\right)$ be the minimal orbit in $\mathbb{P}(\mathbb{S}_{\lambda} V)$. With this notation we have that

$$ \lambda =\left ((d_1+\dots+d_s)^{n_1},(d_2+\dots+d_s)^{n_2-n_1},\dots,d_s^{n_s-n_{s-1}} \right). $$
Hence the Young diagram of $\lambda$ is given by $d_s$ columns of length $n_s$, then $d_{s-1}$ columns of length $n_{s-1}$ and so on up to $d_1$ columns of length $n_1$. For example if $X=\left(\mathbb{F}(1,2,4;\mathbb{C}^5),\mathcal{O}(1,2,2)\right)$, then $\lambda =(5,4,2^2)$ and its Young diagram is 

$$ \ydiagram{5,4,2,2}.$$
Let $\lambda$ be a partition such that it has $d_i$ columns of length $n_i$, with $1 \leq i \leq s$. Consider the catalecticant map determined by $(e^{n_s}) \subset \lambda$ with $e \leq d_s$, i.e. the one that removes the first $e$ columns of length $n_s$ from the diagram of $\lambda$. Then it is easy to see that $\mathbb{S}_{\lambda / (e^{n_s})} V \simeq \mathbb{S}_{\mu_e} V$
where

\begin{align}\label{FormulambdaU}\mu_e = ((d_1+\dots+d_{s-1}+(d_s-e))^{n_1}&,\ (d_2+\dots+d_{s-1}+(d_s-e))^{n_2-n_1},\dots \\ &\dots, (d_{s-1}+(d_s-e))^{n_{s-1}-n_{s-2}},(d_s-e)^{n_s}) \nonumber
\end{align}
i.e. $\lambda$ with the first $e$ columns of length $n_s$ removed.

\begin{algorithm} \label{algoritmo2} In order to get a bound on the $\lambda$-rank using catalecticants, we provide an algorithm that after the computation of a sequence of ranks of ``consecutive'' catalecticant maps returns a lower bound on the $\lambda$-rank of the given tensor. The last registered rank will be a lower bound on the $\lambda$-rank of $t$.

\noindent Before proceeding with the procedure, we need a preparatory fact.

\begin{prop} \label{PropCheServe}
Let $\lambda$ be a partition with $d_i$ columns of length $n_i$, with $i= 1,\dots,s$, and let $t \in \mathbb{S}_{\lambda} V$. If 

$$\operatorname{rk} \mathcal{C}^{\lambda,\left ( \lceil \frac{d_s}{2} \rceil \right)^{n_s}}_t = 1,\ \text{then for any}\ 1 \leq e \leq d_s,\ \text{we have that}\ \operatorname{rk} \mathcal{C}^{\lambda,(e^{n_s})}_t = 1.$$
\end{prop}
\begin{proof}
Let us prove the contraposition, i.e. if there exists an $e \in \{1,\dots,d_s\}$ such that 

$$\operatorname{rk} \mathcal{C}^{\lambda,(e^{n_s})}_t > 1,\ \text{then}\ \operatorname{rk} \mathcal{C}^{\lambda,\left(\lceil \frac{d_s}{2} \rceil \right)^{n_s}}_t > 1. $$
Assume that for a certain $e \in \{1,\dots,d_s\}$ it happens that $\operatorname{rk} \mathcal{C}^{\lambda,(e^{n_s})}_t > 1$. From the hypothesis we can assume that the $\lambda$-rank of $t$ is at least $2$, otherwise it can be easily seen that the rank of $\mathcal{C}^{\lambda,(e^{n_s})}_t$ would be equal to $1$ which is against the hypothesis. Assume then that $t = t_1 + \dots + t_r$ has $\lambda$-rank $r \geq 2$, where every $t_i$ has $\lambda$-rank $1$ and it is written as

$$ t_i = (v_{1,i} \wedge \dots \wedge v_{n_s,i})^{\otimes d_s} \otimes \dots \otimes (v_{1,i} \wedge \dots \wedge v_{n_1,i})^{\otimes d_1}$$
for some vectors $v_{1,i},\dots,v_{n_s,i} \in V$, for all $i = 1,\dots,r$. The catalecticant map $\mathcal{C}^{\lambda,(e^{n_s})}_t$ clearly acts on the first products $(v_{1,i} \wedge \dots \wedge v_{n_s,i})^{\otimes d_s}$ of every $t_i$. Since the rank of such map is at least $2$, we can find at least two points $t_i$ and $t_j$ such that

$$(v_{1,i} \wedge \dots \wedge v_{n_s,i}) \neq (v_{1,j} \wedge \dots \wedge v_{n_s,j})$$
and also such that the images via the catalecticant map of the respective duals $(x_{1,i} \wedge \dots \wedge x_{n_s,i})^{\otimes e}$ and $(x_{1,j} \wedge \dots \wedge x_{n_s,j})^{\otimes e}$ are linearly independent. 
Consider now the catalecticant map 

$$\mathcal{C}^{\lambda,\left(\lceil \frac{d_s}{2} \rceil \right)^{n_s}}_t : \mathbb{S}_{\left(\lceil \frac{d_s}{2} \rceil \right)^{n_s}} V^* \longrightarrow \mathbb{S}_{\mu_{\lceil \frac{d_s}{2} \rceil}} V$$
and the elements $(x_{1,i} \wedge \dots \wedge x_{n_s,i})^{\otimes \lceil \frac{d_s}{2} \rceil}$ and $(x_{1,j} \wedge \dots \wedge x_{n_s,j})^{\otimes \lceil \frac{d_s}{2} \rceil}$. It is clear that they are linearly independent and that their images via $\mathcal{C}^{\lambda,\left(\lceil \frac{d_s}{2} \rceil \right)^{n_s}}_t$  are also linearly independent. Therefore we get that the rank of $\mathcal{C}^{\lambda,\left(\lceil \frac{d_s}{2} \rceil \right)^{n_s}}_t$ is at least $2$. By contraposition we get that if $\operatorname{rk} \mathcal{C}^{\lambda,\left(\lceil \frac{d_s}{2} \rceil \right)^{n_s}}_t = 1,$ then $\operatorname{rk} \mathcal{C}^{\lambda,(e^{n_s})}_t = 1$ for any $1 \leq e \leq d_s$.
This concludes the proof.
\end{proof}

\noindent Now we describe the algorithm which provides a lower bound of the $\lambda$-rank. Let $t \in \mathbb{S}_{\lambda} V$ and consider
the catalecticant map that removes half of the $d_s$ columns of maximal length from the diagram of $\lambda$, rounded up to the next integer if needed. Compute the rank of such a catalecticant map 

$$\mathcal{C}^{\lambda,(\lceil \frac{d}{2} \rceil^{n_s})}_t : \mathbb{S}_{(\lceil \frac{d}{2} \rceil^{n_s})} V^* \longrightarrow \mathbb{S}_{\mu_{\lceil \frac{d}{2} \rceil}} V,$$
where $\mu_{\lceil \frac{d}{2} \rceil}$ denotes the partition as in \eqref{FormulambdaU}. If the rank of the catalecticant is strictly greater than $1$, the algorithm stops and outputs this number. Such a number is a lower bound on the $\lambda$-rank of $t$. Otherwise, by Proposition \ref{PropCheServe} we get that

$$\operatorname{rk} \mathcal{C}^{\lambda,e^{n_s}}_t = 1,$$
for any $1 \leq e \leq d_s$. Hence we can consider the catalecticant with $e = d_s$ and its image will be generated by a unique element up to scalar multiplication. Such image is contained in $\mathbb{S}_{\mu_{d_s}} V$, where $\mu_{d_s}$ denotes the partition with $d_i$ columns of length $n_i$, with $i = 1,\dots,s-1$, accordingly with the notation used in \eqref{FormulambdaU}. Choose a generator $t_1$ of the image of the chosen catalecticant map. Then set $\lambda = \mu_{d_s}$ and $t = t_1$ and repeat the previous steps. \smallskip

\noindent Generalizing, at the $i$-th step we set $\lambda = \lambda_i$, where $\lambda_i$ is the diagram with $d_j$ columns of length $n_j$ for $j = 1,\dots,s-i+1$ , and we set $t = t_i$, where $t_i \in \mathbb{S}_{\lambda_i} V$ is the generator of the one dimensional image obtained from the previous step of the algorithm. Then compute 

$$\operatorname{rk} \left ( \mathcal{C}^{\lambda_i,\left \lceil \frac{d_{s-i+1}}{2} \right \rceil^{n_{s-j}}}_{t_i} \right ).$$ 
If it is strictly greater than $1$ the algorithm outputs this number and it stops. Otherwise consider the catalecticant map that removes all the columns of length $d_{s-i+1}$. Compute the generator $t_{i+1} \in \mathbb{S}_{\lambda_{i+1}} V$ of the image of this last catalecticant map, set $\lambda = \lambda_{i+1}$ and $t = t_{i+1}$, and move to the $(i+1)$-th step. 

\noindent If the rank of every catalecticant map we compute along the procedure is $1$, the algorithm outputs $1$. Obviously the output of the algorithm is an integer greater or equal than $1$ and it is a lower bound on the $\lambda$-rank of $t$. Indeed, preliminary we have

\begin{prop} \label{RanghiConsecutivi}
Consider the flag variety $\mathbb{F}(n_1,\dots,n_s;V)$ embedded with $\mathcal{O}(d_1,\dots,d_s)$ in $\mathbb{P}(\mathbb{S}_{\lambda} V)$. Let $t \in \mathbb{S}_{\lambda} V$ be any point. If $t$ has $\lambda$-rank $1$, then the Algorithm \ref{algoritmo2} outputs a $1$. The converse is true if $d_i \geq 2$ for any $i = 1,\dots,s$. 
\end{prop}

\begin{proof}
Assume that $t$ has $\lambda$-rank $1$, i.e.

$$ t = (v_1 \wedge \dots \wedge v_{n_s})^{\otimes d_s} \otimes \dots \otimes (v_1 \wedge \dots \wedge v_{n_1})^{\otimes d_1} $$
for some $v_i \in V$. The image of the first catalecticant map of the algorithm, i.e. the one determined by $e=1$ and the partition $(1^{n_s})$, is the span of

$$ (v_1 \wedge \dots \wedge v_{n_{s}})^{\otimes d_{s}-1} \otimes (v_1 \wedge \dots \wedge v_{n_{s-1}})^{\otimes d_{s-1}} \otimes \dots \otimes (v_1 \wedge \dots \wedge v_{n_1})^{\otimes d_1}$$
and hence the map has rank $1$ since $t$ is non zero. The same happens for the next steps until $e = \lfloor \frac{d_s}{2} \rfloor$. Therefore consider the catalecticant map that removes all the first $d_s$ columns and consider the only generator $t_{s-1}$ of its image

$$ t_{s-1} = (v_1 \wedge \dots \wedge v_{n_{s-1}})^{\otimes d_{s-1}} \otimes \dots \otimes (v_1 \wedge \dots \wedge v_{n_1})^{\otimes d_1}. $$
At this point set $t = t_{s-1}$ and $\lambda = \mu_{d_s}$ and repeat the previous steps. It is obvious that the algorithm will not stop when computing the rank of any catalecticant map since any such number is equal to $1$.  Therefore it eventually outputs $1$.

\noindent For the converse part, assume that $d_i \geq 2$ for any $i$ and suppose that the output of the algorithm is $1$. Suppose that $t$ has a $\lambda$-rank $r$ decomposition $t = t_1 + \dots + t_r$, with $t_i$ of $\lambda$-rank $1$. We may assume that

$$ t_i = (v_{1,i} \wedge \dots \wedge v_{n_s,i})^{\otimes d_s} \otimes \dots \otimes (v_{1,i} \wedge \dots \wedge v_{n_1,i})^{\otimes d_1}$$
for some vectors $v_{1,i},\dots,v_{n_s,i} \in V$, for all $i = 1,\dots,r$. For any $1 \leq e \leq d_s$, the image of the catalecticant map is one dimensional by Proposition \ref{PropCheServe} and it is contained in $\langle t_1^e,\dots,t_r^e \rangle$, where $t_i^e$ denotes

$$ t_i^e = (v_{1,i} \wedge \dots \wedge v_{n_s,i})^{\otimes d_s-e} \otimes \dots \otimes (v_{1,i} \wedge \dots \wedge v_{n_1,i})^{\otimes d_1}$$
for all $i = 1,\dots, r$. We can assume that the element $(x_{1,1} \wedge \dots \wedge x_{n_s,1})^{\otimes e}$, defined taking the dual elements $x_{i,j}$ to the vectors appearing in the biggest subspace associated to $t_1$, is not apolar to $t$. Hence its image is a generator of the one dimensional image of the respective catalecticant map. If such an element of $\mathbb{S}_{(e^{n_s})} V^*$ is the only one with this property, then we can already conclude that the biggest subspace must be the same for any $t_i$. Otherwise, if we could find another element defined in the same way as $(x_{1,1} \wedge \dots \wedge x_{n_s,1})^{\otimes e}$ but using this time the other $t_i$'s, then its image must be a scalar multiple of the image we have already obtained. Hence a certain linear combination of these two elements is apolar to $t$. On the other hand in the respective images of the two selected elements of $\mathbb{S}_{(e^{n_s})} V^*$ there are also the tensors $t_1^e$ and $t_i^e$ which are linearly independent, unless $e = d_s$ and $t_1^{d_s} = t_i^{d_s}$ which happens only if the points $t_1$ and $t_i$ share the same remaining part of the flag. However, since we are assuming $d_s > 1$, and since the algorithm is settled to pick $\lceil \frac{d_s}{2} \rceil$, we are always considering $e < d_s$ that allows to avoid such a problem. Hence, if the rank of the catalecticant is $1$, we get that all the subspaces are the same. Therefore the image of the catalecticant map is one dimensional and is generated by $t' = t_1^{d_s} + \dots + t_r^{d_s}$. At this point the proof is just a repetition of the previous reasoning until one arrives to get $t_1 = \dots = t_r$, i.e. $t$ has $\lambda$-rank $1$. This concludes the proof. 
\end{proof}

\begin{cor} \label{RanghiConsecutivi2}
Let $t \in \mathbb{S}_{\lambda} V$ and suppose that the output of the Algorithm \ref{algoritmo2} applied to $t$ is $r$. Then $t$ has $\lambda$-rank greater or equal than $r$.
\end{cor}

\noindent We describe briefly the Algorithm \ref{algoritmo2} \index{algorithm!that computes a lower bound on the $\lambda$-rank} in general for any $t \in \mathbb{S}_{\lambda} V$. \medskip

\hrule \noindent {\bf Algorithm \ref{algoritmo2}.} \hrule \medskip

\noindent {\bf Input}: A partition $\lambda$ and an element $t \in \mathbb{S}_{\lambda} V$, where the related minimal orbit inside the projectivization of the space is $X = \left(\mathbb{F}(n_1,\dots,n_s;V),\mathcal{O}(d_1,\dots,d_s)\right)$.

\noindent {\bf Output}: A lower bound of the $\lambda$-rank of $t$.
\smallskip

\begin{enumerate}[nosep, label = \arabic*)]
\item set $i =s$;
\item set $r = 0$;
\item \label{Sicomincia} {\bf if} $i = 0$, {\bf then}
\item $\quad$ {\bf print} $1$ {\bf and exit};
\item set $r = \operatorname{rk}\mathcal{C}^{\lambda,(\lceil \frac{d_i}{2} \rceil^{n_i})}_t$;
\item {\bf if} $r >1$, {\bf then}
\item $\quad$ {\bf print} $r$ and {\bf exit};
\item consider the map $\mathcal{C}^{\lambda,(d_i^{n_i})}_t$ and compute the only generator $t'$ of the image;
\item set $i = i-1$, $t = t'$ and $\lambda = \lambda_{i-1}$ where this last partition is the one with $d_j$ columns of length $n_j$, for $j = 1,\dots,i-1$. Come back to \ref{Sicomincia};
\end{enumerate}
\end{algorithm}
{\hrule \medskip}

\begin{remark} \label{CounterAlg}
Let us highlight that if the Algorithm \ref{algoritmo2} outputs $1$, obviously this does not imply that $t$ has $\lambda$-rank $1$. Indeed, consider as an example of this phenomenon the partition $\lambda = (5,3,1,1)$ and the tensor $t \in \mathbb{S}_{\lambda} V$

$$ t = v_1 \wedge v_2 \wedge v_3 \wedge v_4 \otimes v_1 \wedge v_2 \otimes v_1 + v_1 \wedge v_2 \wedge v_5 \wedge v_6 \otimes v_1 \wedge v_2 \otimes v_1. $$
The output of the Algorithm \ref{algoritmo2} in this case is $1$. This is due to the fact that the two points share the same partial flag $\langle v_1 \rangle \subset \langle v_1,\ v_2\rangle$. Hence the kernel of the first catalecticant map of the algorithm contains in particular the element $x_1 \wedge x_2 \wedge x_3 \wedge x_4 $ $-x_1 \wedge x_2 \wedge x_5 \wedge x_6$. Nonetheless, it is obvious that $t$ has $\lambda$-rank $2$. 
\end{remark}

\begin{example}[Example \ref{EsempioFakeRk2} reprise]
Consider the tensor

$$ t = v_1 \wedge v_2 \wedge v_3 \otimes v_1 \wedge v_2 \otimes v_1 + v_1 \wedge v_2 \wedge v_3 \otimes v_2 \wedge v_3 \otimes v_3 \in \mathbb{P}(\mathbb{S}_{(3,2,1)} \mathbb{C}^4). $$
Following the notation of the Algorithm \ref{algoritmo2}, when $i=3$, then $\lambda_0 = (3,2,1)$ and the first catalecticant map is 

$$ \mathcal{C}^{(3,2,1),(1,1,1)}_t : \mathbb{S}_{(1,1,1)} (\mathbb{C}^4)^* \longrightarrow \mathbb{S}_{(2,1)} \mathbb{C}^4. $$
As we have already remarked this map has rank $1$ so the algorithm can continue. The generator of the image is 

$$ t_3 = v_1 \wedge v_2 \otimes v_1 + v_2 \wedge v_3 \otimes v_3. $$
Set $\lambda_1 = (2,1)$, $i=2$ and $t=t_1$ and restart the algorithm. This time we consider the catalecticant map

$$ \mathcal{C}^{(2,1),(1,1)}_t : \mathbb{S}_{(1,1)} (\mathbb{C}^4)^* \longrightarrow \mathbb{S}_{(1)} \mathbb{C}^4.$$
It is easy to see that this time the rank is $2$. Hence the algorithm stops and the output is $(r_0,r_1) = (1,2)$. As Proposition \ref{RanghiConsecutivi} is going to tell us in a moment, we get that $t$ has $\lambda$-rank at least $2$. 
\end{example}

We now use the procedure to investigate the $\lambda$-rank of a tensor. 

\begin{prop} \label{RanghiConsecutivi}
Consider the flag variety $\mathbb{F}(k_1,\dots,k_s;n)$ embedded with $\mathcal{O}(d_1,\dots,d_s)$ in $\mathbb{P}(\mathbb{S}_{\lambda} \mathbb{C}^n)$. Let $t \in \mathbb{S}_{\lambda} \mathbb{C}^n$ be any point. Then $t$ has $\lambda$-rank $1$ if and only if the output of the Algorithm \ref{algoritmo2} is a sequence of $1'$s of length $s$. As a consequence, if the last integer of the sequence obtained with the Algorithm \ref{algoritmo2} is $r \neq 1$, then $t$ has $\lambda$-rank at least $r$.
\end{prop}

\begin{proof}
Assume that $t$ has $\lambda$-rank $1$, i.e.

$$ t = (v_1 \wedge \dots \wedge v_{k_s})^{\otimes d_s} \otimes \dots \otimes (v_1 \wedge \dots \wedge v_{k_1})^{\otimes d_1}. $$
Then it is easy to see that the catalecticant map computed by the algorithm has rank $1$, where at every step $t$ is set equal to 

$$t_i = (v_1 \wedge \dots \wedge v_{k_i})^{\otimes d_i} \otimes \dots \otimes (v_1 \wedge \dots \wedge v_{k_1})^{\otimes d_1}$$
for all $i $ starting from $s$ to $1$. Hence the output will be a sequence of $1$'s of length $s-1$. 
\noindent On the other hand assume that the output of the Algorithm \ref{algoritmo2} is a sequence of $1$'s of length $s-1$. Suppose $t = p_1 + \dots + p_r$, where every $p_i$ has $\lambda$-rank $1$. Hence each $p_i$ represents a flag of subspaces, each of them repeated a certain number of times. Assume that $r > 1$. Then we have either two alternatives. All the $p_i$'s share the same biggest subspace and this is coherent with the hypothesis. Otherwise there are at least a $p_i$ and a $p_j$ such that the respective two biggest subspaces are different. If this is the case, then, up to choosing a suitable basis, it is easy to see that the rank of the map will be at least $2$. This contradicts the hypothesis and hence all the $p_i$ must share the same biggest subspace. Now repeating this argument at every step of the algorithm one obtains that $t = p_1 = \dots = p_r$, i.e. $t$ has $\lambda$-rank $1$.
\end{proof}

\section{Secant varieties of Flag varieties} \label{sestasez}  \setcounter{equation}{0}\medskip

This section is devoted to the study of the $\lambda$-rank appearing on the second secant variety to a Flag variety. An algorithm collecting all the results obtained will follow.
\medskip

\subsection{The case $\lambda = (2,1^{k-1})$.} \setcounter{equation}{0}\medskip

The first case we consider is given by Flag varieties $\mathbb{F} = \mathbb{F}(1,k;\mathbb{C}^n)$ embedded in $\mathbb{P}(\mathbb{S}_{(2,1^{k-1})} \mathbb{C}^n)$. We will see that such varieties are related to the adjoint varieties $\mathbb{F}(1,k;\mathbb{C}^{k+1})$. Only in this section we will refer to the $(2,1^{k-1})$-rank of a tensor simply with rank.

\begin{definition} Given a non degenerate variety $X \subset \mathbb{P}^N$, we use the following definition for the {\it tangential  variety} of $X$

$$\tau (X) = \bigcup_{p \in X} T_p X$$

\noindent where $T_p X$ denotes the tangent space to $X$ at $p$. 
\end{definition}
In order to study the elements appearing on the second secant variety of $\mathbb{F}$, we have to understand which ranks appear on the tangential variety of $X$. Since $\mathbb{F}$ is homogeneous, we may reduce to study what happens on a single tangent space. By virtue of this fact, choose $p$ as the highest weight vector of this representation, i.e.

\begin{equation} \label{hwtang} p = v_1 \wedge \dots \wedge v_k \otimes v_1 \end{equation}

\noindent for some $v_i \in \mathbb{C}^n$. This element may be represented with the sstd tableau

\begin{center} \begin{ytableau}
1 & 1 \\ 2 \\ \vdots \\ k
\end{ytableau} .\end{center}

\noindent Recall the following classic result.
Let $p=v_1 \wedge \dots \wedge v_k$ be a point of $\mathbb{G}(k, V) \subset \mathbb{P} (\superwedge^k V)$. Then 

\begin{equation} \label{TangGrass} \widehat{T_p \mathbb{G}(k, V)} = \sum_{i=1}^k v_1 \wedge \dots \wedge v_{i-1} \wedge V \wedge v_{i+1} \wedge \dots \wedge v_k. \end{equation}

\begin{prop}
Let $p \in \mathbb{F}$ be the highest weight vector in \eqref{hwtang}. The cone over the tangent space $T_p \mathbb{F}$ to $\mathbb{F}$ at $p$ is the subspace 
\begin{align*}
\langle &v_1 \wedge \dots \wedge v_k \otimes v_1,\ v_1 \wedge \dots \wedge v_{i-1} \wedge v_h \wedge v_{i+1} \wedge \dots \wedge v_k \otimes v_1,\ i \in \{2,\dots,k\}, \\ &v_1 \wedge \dots \wedge v_k \otimes v_h + v_h \wedge v_2 \dots \wedge v_k \otimes v_1,\ h \in \{1,\dots,n \}\rangle \subset \mathbb{S}_{(2,1^{k-1})} \mathbb{C}^n.
\end{align*}
\end{prop}

\begin{proof}
\noindent By definition we have the inclusion $\mathbb{F} \subset \mathbb{G}(1,\mathbb{C}^n) \times \mathbb{G}(k,\mathbb{C}^k)$. By \cite[Prop. 4.4]{freire2019secant}, we have the equality

$$ \widehat{T_p \mathbb{F}} = \left ( \widehat{T_p \mathbb{G}(1,\mathbb{C}^n) \times \mathbb{G}(k,\mathbb{C}^k)} \right ) \cap \mathbb{S}_{(2,1^{k-1})} \mathbb{C}^n$$

\noindent where $\hat{Y}$ denotes the affine cone over the projective variety $Y$. Applying Formula \eqref{TangGrass} we see that $\widehat{T_p \left(\mathbb{G}(1,\mathbb{C}^n) \times \mathbb{G}(k,\mathbb{C}^k)\right)}$ is the subspace 

\begin{align*}
\langle v_1 &\wedge \dots \wedge v_k \otimes v_1,\ v_1 \wedge \dots \wedge v_k \otimes v_2,\dots, v_1 \wedge \dots \wedge v_k \otimes v_n,
\\ &v_1 \wedge \dots \wedge v_{i-1} \wedge v_h \wedge v_{i+1} \wedge \dots \wedge v_k \otimes v_1,\ h \in \{k+1,\dots,n\},\ i \in \{1,\dots,k\} \rangle.
\end{align*}

\noindent It is easy to see that if $i \neq 1$, the elements 
$$v_1 \wedge \dots \wedge v_k \otimes v_1,\ v_1 \wedge \dots \wedge v_{i-1} \wedge v_h \wedge v_{i+1} \wedge \dots \wedge v_k \otimes v_1$$
with $h \in \{k+1,\dots,n\}$ satisfy the relations \eqref{PluckRel}. Indeed, they are, up to the sign, the elements of the Schur module determined by the sstd tableaux

\begin{equation} \label{TangFlag1} \begin{ytableau}
1 & 1 \\ 2 \\ \vdots \\ k
\end{ytableau}\quad  \text{and}\quad  \begin{ytableau}
1 & 1 \\ 2 \\ \vdots \\ \hat{i} \\ \vdots \\ k \\ h
\end{ytableau}\end{equation}

\noindent respectively, where $\hat{i}$ means that $i$ is not appearing in the list. We can see also that the elements

$$ v_1 \wedge \dots \wedge v_k \otimes v_h + v_h \wedge v_2 \dots \wedge v_k \otimes v_1$$

\noindent satisfy the equations \eqref{PluckRel} for any $h=2,\dots,n$ and hence they belong to the module. Indeed they are the elements associated to the sstd tableaux

\begin{equation} \label{TangFlag2} \begin{ytableau}
1 & h \\ 2 \\ \vdots  \\ k
\end{ytableau}\ . \end{equation}

\noindent Consider then the span of the elements of the Schur modules whose associated sstd tableaux is either in \eqref{TangFlag1} or in \eqref{TangFlag2}. Since they are all different, the respective elements of the module are linearly independent. Moreover the number of elements in \eqref{TangFlag1} is $(k-1)(n-k)+1$, and those in \eqref{TangFlag2} are $n-1$ , for a total of $-k^2 + kn + k$. Since the variety $\mathbb{F}$ is smooth of dimension $-k^2+kn+k-1$, we can conclude that $\widehat{T_p \mathbb{F}}$ is the subspace
\begin{align*}
\langle &v_1 \wedge \dots \wedge v_k \otimes v_1,\ v_1 \wedge \dots \wedge v_{i-1} \wedge v_h \wedge v_{i+1} \wedge \dots \wedge v_k \otimes v_1,\ i \in \{2,\dots,k\}, \\ &v_1 \wedge \dots \wedge v_k \otimes v_h + v_h \wedge v_2 \dots \wedge v_k \otimes v_1,\ h \in \{1,\dots,n \}\rangle.
\end{align*}\vskip-0.6cm\end{proof}

\noindent In order to understand which ranks appear in this space, we split the generators of $T_p \mathbb{F}$ in the three following sets:\begin{enumerate}[label = (\arabic*)]
\item \label{item1} $v_1 \wedge \dots \wedge v_k \otimes v_1$,
\item \label{item2} $v_1 \wedge \dots \wedge v_{i-1} \wedge v_h \wedge v_{i+1} \wedge \dots \wedge v_k \otimes v_1$, for $i = 2,\dots,k$ and $h = k+1,\dots,n$,
\item \label{item3} $v_1 \wedge \dots \wedge v_k \otimes v_h + v_h \wedge v_2 \dots \wedge v_k \otimes v_1$ for $h = 2,\dots,n$.
\end{enumerate}
The elements from \ref{item1} and \ref{item2} have rank $1$ and they represent the flags 
$$ \langle v_1 \rangle \subset \langle v_1,\dots,v_k \rangle  $$ 
\noindent and
$$ \langle v_1 \rangle \subset \langle v_1,\dots,v_{i-1},v_h,v_{i+1},v_k \rangle  $$
\noindent respectively. 

\begin{prop} \label{famiglia3} The elements $t$ from \ref{item3} have all rank $1$ for $h = 2,\dots,k$ and they represent the flags
$$\langle v_h \rangle \subset \langle v_1,\dots,v_k \rangle. $$ 
If $h = k+1,\dots,n$ then the corresponding element has rank $2$ and it decomposes as
 $$ t = -\frac{1}{2} (v_1-v_h) \wedge v_2 \wedge \dots \wedge v_k \otimes (v_1-v_h) + \frac{1}{2} (v_1+v_h) \wedge v_2 \wedge \dots \wedge v_k \otimes (v_1+v_h).$$
\end{prop}
\begin{proof} Suppose that $h = 2,\dots,k$. Then $t$ has the form 
$$ v_1 \wedge \dots \wedge v_h \wedge \dots \wedge v_k \otimes v_h$$
and hence it has rank $1$ and it represents the flag
$$\langle v_h \rangle \subset \langle v_1,\dots,v_k \rangle. $$ 
If $h=k+1,\dots,n$, we can compute the kernel
$$ \ker \mathcal{C}_t^{(2,1^{k-1}),(1)} = \langle x_{k+1},\dots,\hat{x_h},\dots,x_n \rangle $$
where $\hat{x_h}$ means that $x_h$ does not appear among the generators. Remark that in general if $t$ has rank $1$, say for instance $t = p$ in \eqref{hwtang}, then the catalecticant map $\mathcal{C}^{(2,1^{k-1}),(1)}_t$ has rank $k$. This implies that if $t = t_1 +\dots + t_r$, where every $t_i$ has rank $1$, we get the inequality

\begin{equation} \operatorname{rk} \mathcal{C}^{(2,1^{k-1}),(1)}_t = \operatorname{rk} \mathcal{C}^{(2,1^{k-1}),(1)}_{t_1 + \dots + t_r} \leq \sum_{i=1}^r \operatorname{rk} \mathcal{C}^{(2,1^{k-1}),(1)}_{t_i} = r \cdot k. \end{equation}

\noindent Since in this case $\operatorname{rk}(\mathcal{C}^{(2,1^{k-1}),(1)}_t) = k+1$, we can already conclude that $t$ has not rank $1$. Hence the decomposition 
$$ t = -\frac{1}{2} (v_1-v_h) \wedge v_2 \wedge \dots \wedge v_k \otimes (v_1-v_h) + \frac{1}{2} (v_1+v_h) \wedge v_2 \wedge \dots \wedge v_k \otimes (v_1+v_h)$$
is minimal and $t$ has rank $2$.
\end{proof}

\begin{remark} \label{primosistema}
The fact that $\operatorname{rk}(\mathcal{C}^{(2,1^{k-1}),(1)}_t) = k+1$ suggests that we can restrict our study to the flag $\mathbb{F}(1,k;\mathbb{C}^{k+1})$ which is an adjoint variety. In this restricted case, in order to find a rank $2$ decomposition of the tensor, we should find at least one product of two distinct linear forms inside $\ker\mathcal{C}_t^{(2,1^{k-1}),(2)}$. Such linear forms are the equations of the two $k$-dimensional linear spaces in $\mathbb{C}^{k+1}$ of the two flags associated to the decomposition. We can see that
$$ (x_1-x_h)(x_1+x_h) \in \ker \mathcal{C}^{(2,1^{k-1}),(2)} $$
are the linear forms we are looking for and the respective $k$-dimensional linear spaces are
$$ v(x_1-x_h) = \langle v_1 + v_h,v_2,\dots,v_k \rangle $$
and
$$ v(x_1+x_h) = \langle v_1 - v_h,v_2,\dots,v_k \rangle. $$
\end{remark} 
\smallskip

\noindent Now we have to study the possible sums of elements from the three different sets. 

\begin{remark} The sum of two elements from \ref{item1} and \ref{item2} gives
\begin{align*} a\cdot v_1 \wedge \dots \wedge v_k \otimes v_1 + &b\cdot v_1 \wedge \dots \wedge v_{i-1} \wedge v_h \wedge v_{i+1} \wedge \dots \wedge v_k \otimes v_1 = \\ &v_1 \wedge \dots \wedge v_{i-1} \wedge (a \cdot v_i + b \cdot v_h) \wedge v_{i+1} \wedge \dots \wedge v_k \otimes v_1
\end{align*}
which has rank $1$. 
\end{remark}

\begin{remark}
The sum of two elements from \ref{item1} and \ref{item3} with $h = 2,\dots,k$ turns out to be
\begin{align*} a\cdot v_1 \wedge \dots \wedge v_k \otimes v_1 + b \cdot v_1 \wedge \dots \wedge v_k \otimes v_h  = (v_1 + v_h) \wedge \dots \wedge v_k \otimes (v_1 + v_h)
\end{align*}
which has rank $1$, while if $h = k+1,\dots,n$ is
\begin{align*} & a\cdot v_1 \wedge \dots \wedge v_k \otimes v_1 +b \cdot ( v_1 \wedge \dots \wedge v_k \otimes v_h + v_h \wedge v_2 \dots \wedge v_k \otimes v_1) = \\ & b \cdot v_1 \wedge \dots \wedge v_k \otimes (\frac{a}{2} \cdot v_1 + b \cdot v_h) + (\frac{a}{2} \cdot v_1 + b \cdot v_h) \wedge v_2 \dots \wedge v_k \otimes v_1
\end{align*}
which is again an element of the set \ref{item3}. 
\end{remark}

\noindent Now we focus on the case \ref{item2} $+$ \ref{item3}.
\begin{prop}
For any $v_j \in \mathbb{C}^n$, the sum of two elements from \ref{item2} and \ref{item3}, i.e. the tensor 
\begin{align} \label{2+3}
t = a \cdot v_1 \wedge \dots \wedge v_{i-1} \wedge &v_j \wedge v_{i+1} \wedge \dots \wedge v_k \otimes v_1 + \\ &+ b \cdot (v_1 \wedge \dots \wedge v_k \otimes v_h + v_h \wedge v_2 \dots \wedge v_k \otimes v_1), \nonumber
\end{align}
has 
\begin{enumerate}[label = (\roman*)]
\item \label{item11} rank $2$ if $v_j$ and $v_h$ are not multiple and $v_h \in \langle v_2,\dots,v_k \rangle$
\item \label{item12} rank $3$ if $v_j$ and $v_h$ are not multiple and $v_h \not\in \langle v_2,\dots,v_k \rangle$
\item \label{item13} rank $2$ if $v_j$ and $v_h$ are multiple.
\end{enumerate}
\end{prop}

\begin{proof}
Assume at first that $v_j$ and $v_h$ are not multiple each other. If $v_h \in \langle v_2,\dots,v_k \rangle$, the sum \eqref{2+3} reduces to
\begin{equation} \label{2+3.1}a \cdot v_1 \wedge \dots \wedge v_{i-1} \wedge v_j \wedge v_{i+1} \wedge \dots \wedge v_k \otimes v_1 + b \cdot v_1 \wedge \dots \wedge v_k \otimes v_h. \end{equation}
We claim that this element has rank $2$. Indeed if the tensor in \eqref{2+3.1} has rank $1$, then the catalecticant map 
$$\mathcal{C}_t^{(2,1^{k-1}),(1)} : \mathbb{S}_{(1)} (\mathbb{C}^n)^* \rightarrow \mathbb{S}_{(2,1^{k-1})/(1)} \mathbb{C}^n$$
 has rank $k$ as already discussed in the proof of Proposition \ref{famiglia3}. Since the catalecticant map in this case has rank $k+1$, we can already conclude that the decomposition in \eqref{2+3.1} is minimal. This proves \ref{item11}.

\noindent Assume now that $v_h \not \in \langle v_2,\dots,v_k \rangle$. In this case we use the catalecticant map 
$$\mathcal{C}_t^{(2,1^{k-1}),(2)} : \mathbb{S}_{(2)} (\mathbb{C}^n)^* \rightarrow \mathbb{S}_{(2,1^{k-1})/(2)} \mathbb{C}^n$$
to compute the rank of the element. At first remark that if $t$ is an element of rank $1$, then the rank of this catalecticant map is $k$. Indeed for instance if $t = p$, then the only elements of $\mathbb{S}_{(2)} (\mathbb{C}^n)^*$ which does not kill $p$ and whose images via the catalecticant map are linearly independent are
$$ x_1^2, x_1x_2, \dots,x_1x_k$$
which are exactly $k$. This implies that if $t = t_1 + \dots + t_r$ has rank $r$, then
\begin{equation} \label{rankbound} \operatorname{rk} \mathcal{C}^{(2,1^{k-1}),(2)}_t = \operatorname{rk} \mathcal{C}^{(2,1^{k-1}),(2)}_{t_1 + \dots + t_r} \leq \sum_{i=1}^r \operatorname{rk} \mathcal{C}^{(2,1^{k-1}),(2)}_{t_i} = r \cdot k. \end{equation}
In the instance of a tensor $t$ like \eqref{2+3} one gets that the kernel of $\mathcal{C}^{(2,1^{k-1}),(2)}_t$ is the subspace
$$\ker \mathcal{C}^{(2,1^{k-1}),(2)}_t = \langle x_mx_n,\ \text{where either}\ (m,n)=(h,h)\ \text{or}\ m,n \neq 1,h \rangle $$
i.e. the elements of $\mathbb{S}_{(2)} (\mathbb{C}^n)^*$ not killing $t$ are all the ones in the span
$$ \langle x_1x_h,\dots,x_kx_h,x_1x_2,\dots,x_1x_k,x_1x_j,x_1^2\rangle. $$
This means that $\operatorname{rk} \mathcal{C}^{(2,1^{k-1}),(2)}_t = 2k+1$. This implies by \eqref{rankbound} that $t$ has rank at least $3$. However, by Proposition \ref{famiglia3}, the element $t$ is written as a sum of $3$ rank $1$ elements and hence its rank is $3$.

\noindent Finally assume that $v_j$ and $v_h$ are multiples. The element in \eqref{2+3} reduces to 
\begin{align*} 
t = a \cdot v_1 \wedge \dots \wedge v_{i-1} \wedge &v_h \wedge v_{i+1} \wedge \dots \wedge v_k \otimes v_1 + \\ &+ b \cdot (v_1 \wedge \dots \wedge v_k \otimes v_h + v_h \wedge v_2 \dots \wedge v_k \otimes v_1). \nonumber
\end{align*}
We obtain again that $\operatorname{rk} \mathcal{C}^{(2,1^{k-1}),(1)}_t = k+1$ and hence $t$ has not rank $1$. One can see that $t$ can be written as
\begin{align*} 
t = &(v_1-v_h) \wedge \dots \wedge v_{i-1} \wedge (v_i-v_h) \wedge v_{i+1} \wedge \dots \wedge v_k \otimes (v_1-v_h) + \\ &+ (v_1+v_h) \wedge \dots \wedge v_{i-1} \wedge (v_i+v_h) \wedge v_{i+1} \wedge \dots \wedge v_k \otimes (v_1+v_h)
\end{align*}
and hence $t$ has rank $2$. Note that up to change of coordinates this is the tensor described in Remark \ref{primosistema}. This concludes the proof.
\end{proof}

\noindent We collect the elements we found in a table.

\begin{table}[ht]\begin{center}
\begin{tabular}{| c | c | c | c |}
\hline
$(2,1^{k-1})$-rank & $\operatorname{rk} \mathcal{C}_t^{(2,1^{k-1}),(1)}$ & $\operatorname{rk} \mathcal{C}_t^{(2,1^{k-1}),(2)}$ & Notes \\ \hline
1 & $k$ & $k$ & sets $(1)$, $(2)$ and $(3)$ \\ \hline
2 & $k+1$ & $2k-1$ & set $(3)$ \\ \hline
3 & $k+2$ & $2k+1$ & $(2) + (3)$ \\ \hline
\end{tabular}
\caption{The ranks appearing on the tangential variety to $\mathbb{F}$.}  \label{tabella1}\end{center}
\end{table}

\noindent Let us study now the elements of $\mathbb{S}_{(2,1^{k-1})} V$ lying on a secant line to $\mathbb{F}$. Such elements can be written as
$$  v_1 \wedge \dots \wedge v_k \otimes v_1 + w_1 \wedge \dots \wedge w_k \otimes w_1 $$
so that their rank is at most $2$. Remark that letting the group $SL(n)$ act on this element, the rank and the numbers $\dim \langle v_1,\dots,v_k \rangle \cap \langle w_1,\dots,w_k \rangle$ and $\dim \langle v_1,\dots,v_k \rangle + \langle w_1,\dots,w_k \rangle$ are preserved. In particular we may pick as a representative of the orbit of the element in the previous formula 

\begin{equation} \label{rango2} t=v_1 \wedge \dots \wedge v_h \wedge v_{h+1} \wedge \dots \wedge v_k \otimes v_i + v_1 \wedge \dots \wedge v_h \wedge v_{k+1} \wedge \dots \wedge v_{2k-h} \otimes v_j \end{equation}
in which the intersection of the $k$-dimensional spaces of the flags is explicit, i.e.

$$\langle v_1,\dots,v_k \rangle \cap \langle v_1,\dots,v_h,v_{k+1},\dots,v_{2k-h} \rangle = \langle v_1,\dots,v_h \rangle .$$ 
The vectors $v_i$ and $v_j$ appearing after the tensor products in the first and second addend are one of the generators of the spaces $\langle v_1,\dots,v_k \rangle$ and $\langle v_1, \dots,v_h,v_{k+1},\dots,v_{2k-h}\rangle$ respectively. Note that if both $v_i$ and $v_j$ belong to the intersection of the $k$-dimensional subspaces of the flags, then they can be the same vector up to scalar multiplication. Hence we may distinguish the elements on secant lines to $\mathbb{F}$ using only two invariants: the dimension of the intersection of the $k$-dimensional subspaces of the two flags and whether the equality $\langle v_i \rangle = \langle v_j \rangle$ holds or not. In terms of Schur apolarity action, we can use the catalecticant maps to determine which orbit we are studying. Specifically the map

$$ \mathcal{C}_t^{(2,1^{k-1}),(1)} : \mathbb{S}_{(1)} (\mathbb{C}^n)^* \longrightarrow \mathbb{S}_{(2,1^{k-1})/(1)} \mathbb{C}^n $$
will give us information about the dimension of the intersection. Indeed consider a rank $2$ element $t$ as in \eqref{rango2} and denote with $\{x_i,\ i=1,\dots,n\}$ the dual basis of the $\{v_i,\ i=1,\dots,n\}$. Then the image of $x_i$ can be either $0$ if $x_i = x_{2k-h+1},\dots,x_n$, or non zero if $x_i = x_1,\dots,x_{2k-h}$. It is easy to see that in this latter case all the images that we get are linearly independent as elements of $\mathbb{S}_{(2,1^{k-1})/(1)}\mathbb{C}^n$. Hence the rank of the catalecticant map is equal to the dimension of the sum of the two $k$-dimensional subspaces of the two flags involved. Once that this number is fixed, the rank of the catalecticant
$$ \mathcal{C}_t^{(2,1^{k-1}),(2)} : \mathbb{S}_{(2)} (\mathbb{C}^n)^* \longrightarrow \mathbb{S}_{(2,1^{k-1})/(2)} \mathbb{C}^n $$
will help us on discriminating whether $\langle v_i \rangle = \langle v_j \rangle$ holds or not. Indeed by the definition of the Schur apolarity action, the element $x_p x_q \in \mathbb{S}_{(2)}(\mathbb{C}^n)^*$, which can be written as $x_p \otimes x_q + x_q \otimes x_p$, is applied in both the factors of the tensor product $\superwedge^k \mathbb{C}^n \otimes \superwedge^1 \mathbb{C}^n$ in which $\mathbb{S}_{(2,1^{k-1})}\mathbb{C}^n$ is contained. Hence the fact that $\langle v_i \rangle = \langle v_j \rangle$ is true or not will change the rank of this catalecticant map.

\noindent We give in the following table a classification of all the possible orbits depending on the invariants we have mentioned. Let us denote briefly $\dim \langle v_1,\dots,v_k \rangle \cap \langle w_1,\dots,w_k \rangle$ just with $\dim V \cap W$.

\begin{table}[ht] \begin{center}
\begin{tabular}{| c | c | c | c | c |}
\hline
$\dim V \cap W$ & $\langle v_1 \rangle = \langle w_1 \rangle$ & $(2,1^{k-1})$-rank & $\operatorname{rk} \mathcal{C}_t^{(2,1^{k-1}),(1)}$ & $\operatorname{rk} \mathcal{C}_t^{(2,1^{k-1}),(2)}$ \\ \hline
$k$ & True/False & 1 & $k$ & $k$ \\ \hline
$k-1$ & True & 1 & $k$ & $k$ \\ \hline
$k-1$ & False & 2 & $k+1$ & $2k-1$ \\ \hline
$\vdots$ & & & & \\ \hline
$h$ & False & 2 & $2k-h$ & $2k$ \\ \hline
$h$ & True & 2 & $2k-h$ & $2k-h$ \\ \hline
$\vdots$ & & & & \\ \hline
$0$ & False & 2 & $2k$ & $2k$ \\ \hline
\end{tabular}
\caption{Orbits of points on a secant line to $\mathbb{F}$, where $h = 0,\dots,k$.} \end{center}
\end{table}
\vskip-0.5cm

\noindent The results obtained so far can be collected in the following algorithm.

\begin{algorithm} The following algorithm distinguishes the rank of the elements of border rank at most $2$.
\vskip0.5cm

\noindent {\bf Input}: An element $t \in \widehat{\sigma_2(\mathbb{F}(1,k;\mathbb{C}^n))} \subset \mathbb{S}_{(2,1^{k-1})} \mathbb{C}^n$.

\noindent {\bf Output}: If the border rank of $t$ is less or equal to $2$, it returns the rank of $t$.

\begin{enumerate}[nosep]
\item[1:] compute $(r_1,r_2,r_3) = \left (\operatorname{rk} \mathcal{C}_t^{(2,1^{k-1}),(1^k)},\operatorname{rk} \mathcal{C}_t^{(2,1^{k-1}),(1)},\operatorname{rk} \mathcal{C}_t^{(2,1^{k-1}),(2)} \right)$
\item[2:] {\bf if} $r_1 = 1$ {\bf then}
\item[3:] $\quad$ $t$ has both border rank and rank equal to $1$, {\bf exit};
\item[4:] {\bf else if} $r_1 \geq 3$ {\bf then}
\item[5:] $\quad$ $t$ has border rank at least $3$, {\bf exit};
\item[6:] {\bf else if} $r_1 = 2$ {\bf then}
\item[7:] $\quad$ $t$ has border rank $2$ and
\item[8:] $\quad$ {\bf if} $(r_1,r_2,r_3) = (2,k+2,2k+1)$ {\bf then}
\item[9:] $\quad$ $t$ has rank $3$ and it is the element given by \ref{item2}$+$\ref{item3} in Table \ref{tabella1};
\item[10:] $\quad$ {\bf else if} $(r_1,r_2,r_3) = (2,2k-h,2k)$ {\bf then}
\item[11:] $\quad$ $t$ has rank $2$ and it is in the orbit with $\dim V \cap W = h$ and $\langle v_i \rangle \neq \langle v_j \rangle$,
\item[12:] $\quad$ {\bf else if} $(r_1,r_2,r_3) = (2,2k-h,2k-h)$ {\bf then}
\item[13:] $\quad$ $t$ has rank $2$ and it is in the orbit with $\dim V \cap W = h$ and $\langle v_i \rangle = \langle v_j \rangle$
\item[14:] $\quad$ {\bf end if}
\item[15:] {\bf end if}
\item[16:] {\bf end}
\end{enumerate}
\end{algorithm}

\begin{remark}
Note that this is not a complete classification of the orbits appearing on $\sigma_2(X)$ by the action of $SL(V)$. For this purpose one has to make a more specific distinction of the orbits related to secant lines to $\mathbb{F}$. In particular one has to discriminate whether the lines $\langle v_i \rangle$ and $\langle v_j \rangle$ both belong to the intersection of the two $k$-dimensional subspaces, only one of them belongs to this intersection and eventually none of them belong to it.
\end{remark}

\noindent As a conclusion of the discussion of this section, we obtain the following result. For a given non degenerate irreducible projective variety $X \subset \mathbb{P}^N$, for $s \geq r$ we use the notation

$$ \sigma_{r,s} (X) := \{ p \in \sigma_r(X): r_X(p)=s \}. $$
\smallskip

\begin{cor}
Let $\mathbb{F} = \mathbb{F}(1,k;n) $ embedded with $\mathcal{O}(1,1)$ in $\mathbb{P}(\mathbb{S}_{(2,1^{k-1})} \mathbb{C}^n)$. Then we have

$$ \sigma_2(\mathbb{F}) \setminus \mathbb{F} = \sigma_{2,2} (\mathbb{F}) \cup \sigma_{2,3} (\mathbb{F}). $$
\end{cor}
\bigskip

\section*{Acknowledgements}
I thank Giorgio Ottaviani for the help and useful comments, and Alessandra Bernardi for suggesting me the topic and for the support. I would like to thanks also Jan Draisma and the referees for their useful comments and remarks. The author is partly supported by GNSAGA of INDAM. \bigskip

\noindent Contact: reynaldo.staffolani@unitn.it, Ph.D. student at University of Trento.

\bibliographystyle{abbrv}
\bibliography{Schur_apolarity}

\end{document}